\documentclass[10pt,a4paper]{article}
\usepackage[T1]{fontenc}
\usepackage[utf8]{inputenc}
\usepackage{authblk}
\usepackage{amsmath,amssymb,amsthm,esint,bm}
\usepackage{mathrsfs}
\usepackage{bookmark}
\usepackage{amsmath}
\allowdisplaybreaks[3]

\newtheorem{definition}{Definition}[section]
\newtheorem{theorem}[definition]{Theorem}
\newtheorem{lemma}[definition]{Lemma}
\newtheorem{proposition}[definition]{Proposition}
\newtheorem{corollary}[definition]{Corollary}
\theoremstyle{remark}
\newtheorem{remark}[definition]{Remark}
\numberwithin{equation}{section}

\newcommand{\RNum}[1]{\uppercase\expandafter{\romannumeral #1\relax}}

\title{Riesz potential estimates for double obstacle problems with Orlicz growth}

\author[a]{Qi Xiong}
\author[b]{Zhenqiu Zhang\thanks{Corresponding author.}}
\author[c]{Lingwei Ma}

\affil[a]{School of Mathematics, Southwest Jiaotong University, Chengdu, Sichuan, 610031, P.R. China}
\affil[b]{School of Mathematical Sciences and LPMC, Nankai University, Tianjin, 300071, P.R. China}
\affil[c]{School of Mathematical Sciences, Tianjin Normal University, Tianjin, 300387, P.R. China}

\date{\today}

\usepackage{hyperref}
\begin{document}
\maketitle
\footnotetext[1]{E-mail: xq@swjtu.edu.cn(Q. Xiong),  zqzhang@nankai.edu.cn (Z. Zhang),mlw1103@163.com (L. Ma).}

\maketitle
\begin{abstract}
In this paper, we consider the solutions to the non-homogeneous   double  obstacle problems  with  Orlicz growth involving measure data. After establishing the existence of the solutions to this problem in the Orlicz-Sobolev space,  we  derive  a   pointwise gradient estimate for these solutions by Riesz potential,   which leads to the result on the $C^1$ regularity criterion.
\\


\end{abstract}


\section{Introduction and main results}\label{section1}
\ \ \  In this paper, we consider  the non-homogeneous  double obstacle problems with Orlicz growth and they are  related to  measure data problems of the type
\begin{equation}\label{1.1}
  -\operatorname{div}\left({a}(x,Du)\right) =\mu \quad\quad\mbox{in}\ \ \ \Omega \\[0.05cm],
\end{equation} 
where $ \Omega\subseteq \mathbb{R}^n, n\geqslant2 $ is a bounded open set and $\mu \in \mathcal{M}_{b}(\Omega)$, where $\mathcal{M}_{b}(\Omega)$  is the set of signed Radon measures $\mu$ for which $|\mu|(\Omega)$ is finite and  here we  denote by $|\mu|$ the total variation of $\mu$.
Moreover we
assume that $\mu(\mathbb{R}^n \backslash \Omega)=0$ and $a=a(x,\eta): \Omega \times\mathbb{R}^n \rightarrow \mathbb{R}^n$ is measurable for each $x\in \Omega$ and differentiable for almost every $\eta \in \mathbb{R}^n$ and there exist constants $0<l\leqslant 1 \leqslant L<+\infty$ such that for all $x \in \Omega,\eta,\lambda \in \mathbb{R}^n$,
\begin{eqnarray}\label{a(x)1}
  \left\{\begin{array}{r@{}c@{}ll}
&&D_{\eta} a(x,\eta )\lambda \cdot \lambda \geqslant l\dfrac{g(|\eta|)}{|\eta|}|\lambda|^2 \,, \\[0.05cm]
&&|a(x,\eta)|+|\eta||D_{\eta} a(x,\eta )|\leqslant Lg(|\eta|)\,, \\[0.05cm]
  \end{array}\right.
\end{eqnarray}
where $D_{\eta}$ denotes the differentiation in $\eta$ and $g(t) : [0,+\infty)\rightarrow [0,+\infty)$ satisfies
\begin{eqnarray}\label{a(x)3}
  \left\{\begin{array}{r@{}c@{}ll}
&&g(t)=0 \ \ \ \Leftrightarrow \ \ \   t=0 \,, \\[0.05cm]
&&g(\cdot)\in C^{1}(\mathbb{R}^+)\,, \\[0.05cm]
&& 1\leq i_{g}=: \inf_{t>0}\frac{tg'(t)}{g(t)}\leq \sup_{t>0}\frac{tg'(t)}{g(t)}=:s_{g}<\infty. \, \\[0.05cm]
  \end{array}\right.
\end{eqnarray}
We define
\begin{equation}\label{g}
G(t):= \int_0^tg(\tau)\operatorname{d}\!\tau \ \ \ \mbox{for}\ \ t\geq0.
\end{equation}
It's obvious that  $G(t)$ is convex and  strictly  increasing. We stress that we impose the Orlicz growth condition of $a(\cdot, \cdot)$ naturally covering the case of (possibly weighted) $p$-Laplacian when $G(t) = t^p$ with $p\geqslant2$, together with $p$-growth condition (see \cite{dm00}) when
 $$G(t)=\int_{0}^{t}(\mu+s^2)^{\frac{p-2}{2}} s ds$$ with $\mu\geqslant0, p\geqslant2$. 
 This type of problem is arising in the fields of fluid dynamics, magnetism, and mechanics, as illustrated in reference \cite{bl}.
Lieberman \cite{l1} initially introduced this class of elliptic equations and demonstrated the $C^{\alpha}$- and $C^{1,\alpha}$-regularity of their solutions. 
Since then, significant advancements have been made in the theory of regularity for such equations, as documented in the references \cite{bm20,cm16,cm17,cm15,rt1}.

 The obstacle condition that we impose on the solutions is  of the form $\psi_2 \geq u \geq \psi_1$ a.e.  in $\Omega$, where $\psi_1, \psi_2 \in W^{1,G}(\Omega)\cap W^{2,1}(\Omega)$ are  given functions which satisfy  $\operatorname{div}\left({a}(x,D\psi_1)\right) \in L^{1}_{loc}(\Omega), $ $ \operatorname{div}\left({a}(x,D\psi_2)\right) $ $ \in L^{1}_{loc}(\Omega)$  and $G$ is defined as \eqref{g}. If we consider an inhomogeneity  $f \in L^{1}(\Omega)\cap (W^{1,G}(\Omega))'$, where $(W^{1,G}(\Omega))'$ is the dual of $W^{1,G}(\Omega)$, the obstacle problem is characterized by the variational inequality
\begin{equation}\label{fjd}
\int_{\Omega} a(x,Du)\cdot D(v-u)dx \geq \int_{\Omega} f(v-u) dx
\end{equation}
for all functions $v \in u+W_0^{1,G}(\Omega)$ with $\psi_2 \geq v \geq \psi_1 $ a.e. in $\Omega$. The work in \cite{rt1} has confirmed the existence and uniqueness of  weak solution to the variational inequality \eqref{fjd}. Nevertheless, our attention is directed towards solutions for double obstacle problems with measure data, with the specific aim of substituting the inhomogeneity $f$ with a bounded Radon measure $\mu$. In this case, we adopt the notion of a limit of approximating  solutions as introduced in \cite{s1},  the double obstacle problems can be obtained through approximation using solutions to variational inequalities \eqref{fjd}, for a precise definition, please refer to Definition \ref{opdy}.

In this paper, we are interested in the precise transfer of regularity properties from the data $\mu$ and obstacle functions $\psi_1, \psi_2$ to the solution $u$ by using Riesz potentials.  Potential theory is essentially a part of regularity theory of partial differential equations and its aim is to provide pointwise  estimates and fine properties of solutions for nonlinear equations, which extend in a most natural way the classical ones valid for linear equations via the representation formula.
These pointwise estimates provide a unified approach to obtain the norm bounds for solutions in a wide range of function spaces. As a result, some regularity properties for solutions can be established, such as H$\ddot{o}$lder continuity, Calder\'on-Zygmund estimates and so on.
  Starting from the  fundamental results of  Kilpel$\ddot{a}$inen Mal$\acute{y}$ \cite{km5,km6}, who established pointwise estimates for solutions to  the nonlinear equations of $p$-Laplace type by the  nonlinear  Wolff potential:
$$c_1W^{\mu}_{1,p}(x,R)\leqslant u(x) \leqslant c_2W^{\mu}_{1,p}(x,R)+c_2 \inf_{B_R(x)}u,$$
where the nonlinear Wolff potential of $\mu$ is defined as
\begin{equation*}
W^{\mu}_{\beta,p}(x,R):=\int_0^R\left( \frac{|\mu|(B_{\rho}(x))}{\rho^{n-\beta p}}\right) ^{1/(p-1)}\frac{\operatorname{d}\!\rho}{\rho}
\end{equation*}
for parameters $\beta \in (0,n]$  and $p>1$. Subsequently, these results were extended to a general setting by Trudinger and Wang \cite{tw7,tw8}  using a different  approach.
Furthermore,  Mingione \cite{m9} first obtained Riesz potential estimates for gradient of solutions to nonlinear elliptic equations with linear growth ($p=2$) :
$$|Du(x)|\leqslant c\textbf{I}^{|\mu|}_{1}(x,R)+c\fint_{B_R(x)}(|Du|+s)dy,$$
where the Riesz potential are  defined by
\begin{equation*}
\textbf{\RNum{1}}^{|\mu|}_{\beta}(x,R):=\int_{0}^{R}  \frac{|\mu|(B_{\rho}(x))}{\rho^{n-\beta }}\frac{\operatorname{d}\!\rho}{\rho}.
\end{equation*}
Its form is essentially the same as the classical one valid for the Poisson equation.  In \cite{dm10}, Duzaar and Mingione proved pointwise gradient  estimates for the $p$-growth problems with $p\geqslant 2$ by Wolff potential.  In addition,
pointwise and oscillation  estimates for solutions and the gradient of solutions  by Wolff potentials have been achieved by  Duzaar  and Mingione \cite{dm00,dm10,km12}.

In \cite{km00}, Mingione proved a somewhat  surprising result by obtaining Riesz potential estimates for the gradient for
the $p$-growth problems with $p\geqslant 2$.
Indeed, the Riesz potential estimates directly imply the Wolff potential estimates for $p\geqslant2$, for more details, see \cite{km00}. Subsequently,   Kuusi  and Mingione \cite{km01}  obtained  oscillation estimates of solutions  using Riesz potential.  The extension of these gradient potential estimates includes parabolic equations \cite{kmpw} and  elliptic systems \cite{kmz}.
 Moreover, Scheven \cite{s1,s28} first obtained some potential estimates for the nonlinear elliptic  obstacle problems with $p$-growth. For more results, please see \cite{bm20,dm10,dm11,dm01,m9,mz1,xiao,xiong5}.

As for the elliptic equations with Orlicz growth, Baroni \cite{b13} established  Riesz potential estimates  for gradient of  solutions to  elliptic equations with constant coefficients.
Later, Xiong, Zhang and Ma \cite{xiong3} extended  the result to  equations with Dini-$BMO$ coefficients. The  Wolff potential estimates for elliptic systems was eatablished in \cite{cy1} and for elliptic obstacle problems was obtained in \cite{xiong1,xiong2}.

The aim of this work is to  prove the Riesz potential estimates for the elliptic double obstacle problems  with Dini-$BMO$ coefficients.   The main difficulty arises within the interplay between measure and two obstacles; to overcome this, we establish some suitable comparison estimates to transfer the double obstacle problems to the homogeneous  equation, then we deduce excess decay estimates for solutions of double obstacle problems, then iterating resulting estimates to obtain potential estimates.

Next, we summarize our main results. We begin by presenting some  definitions, notations and assumptions.
\begin{definition}
A function $G :[0,+\infty)\rightarrow[0,+\infty)$ is called a Young function if it is convex and $G(0)=0$.
\end{definition}
\begin{definition}
Assume that $G$ is a Young function,  the Orlicz class $K^{G}(\Omega)$ is the set of all measurable functions $u : \Omega\rightarrow\mathbb{R}$ satisfying
\begin{equation*}
\int_\Omega G(|u|) \operatorname{d}\!\xi < \infty.\nonumber
\end{equation*}
The Orlicz space $L^{G}(\Omega)$ is the linear hull  of the Orlicz class  $K^{G}(\Omega)$ with the Luxemburg norm
\begin{equation*}
\Vert u \Vert_{L^G(\Omega)}:=\inf\left\lbrace \alpha>0: \ \ \int_{\Omega}G\left(\frac{|u|}{\alpha} \right) \operatorname{d}\!\xi \leqslant1\right\rbrace .
\end{equation*}
Furthermore, the Orlicz-Sobolev space $W^{1,G}(\Omega)$ is defined as
\begin{equation*}
W^{1,G}(\Omega)=\left\lbrace  u\in L^{G}(\Omega)\cap W^{1,1}(\Omega) \ \vert \ Du\in L^{G}(\Omega)\right\rbrace.\nonumber
\end{equation*}
The space $W^{1,G}(\Omega)$, equipped with the norm
$\Vert u \Vert_{W^{1,G}(\Omega)}:=\Vert u \Vert_{L^G(\Omega)}+\Vert Du \Vert_{L^G(\Omega)},$ is a Banach space. Clearly, $W^{1,G}(\Omega)=W^{1,p}(\Omega)$, the standard Sobolev space, if $G(t)=t^p$ with $p\geqslant1$.
\end{definition}

 The subspace $W_{0}^{1,G}(\Omega)$ is the closure of $C_{0}^{\infty}(\Omega)$ in $W^{1,G}(\Omega)$. The above properties about Orlicz space can be found in \cite{hphp1,rr28}.

For every $k>0$ we let
\begin{equation}\label{tks}
T_{k}(s):=
\left\{\begin{array}{r@{\ \ }c@{\ \ }ll}
s\ \ \ \ \ \ \ \ if\ \ |s|\leqslant k\,, \\[0.05cm]
k\ sgn(s)\ \ \ \ \ \ \ if\ \ |s|> k\,. \\[0.05cm]
\end{array}\right.
\end{equation}

Moreover, for given Dirichlet boundary data $h\in W^{1,G}(\Omega)$, we define
$$\mathcal{T}^{1,G}_{h}(\Omega):=\left\lbrace u: \Omega\rightarrow \mathbb{R} \ measurable: T_{k}(u-h)\in W_{0}^{1,G}(\Omega) \ \ for \ all \ k>0\right\rbrace. $$

We now give the definition of approximable solutions.

\begin{definition}\label{opdy}
Suppose that two obstacle functions $\psi_1, \psi_2 \in W^{1,G}(\Omega)$, measure data $\mu \in \mathcal{M}_{b}(\Omega)$ and boundary data $h \in W^{1,G}(\Omega)$ with $\psi_2 \geq h\geq \psi_1$ a.e. are given. We say that $u \in \mathcal{T}^{1,G}_{h}(\Omega)$ with $\psi_2 \geq u \geq    \psi_1$ a.e. in $\Omega$ is  a limit of approximating solutions of the obstacle problem $OP(\psi_1 ; \psi_2; \mu)$ if there exist functions
$$f_{i} \in (W^{1,G}(\Omega))'\cap L^{1}(\Omega)\ \  with\ \  f_{i}\stackrel{\ast}\rightharpoonup \mu \ in \ \mathcal{M}_{b}(\Omega) \ \ as \ i\rightarrow+\infty$$
 satisfies
$$\limsup_{i\rightarrow+\infty}\int_{B_R(x_0)}|f_i|dx\leqslant|\mu|(\overline{B_R(x_0)}),$$
and solutions $u_{i}\in W^{1,G}(\Omega)$ with $\psi_2 \geqslant u_{i}\geqslant \psi_1$ of the variational inequalities
\begin{equation}\label{opdy1}
\int_{\Omega}a(x,Du_{i})\cdot D(v-u_{i})dx\geqslant \int_{\Omega}f_{i}(v-u_{i})dx
\end{equation}
for $\forall \ v \in u_{i}+W_{0}^{1,G}(\Omega)$ with $\psi_2 \geqslant v\geqslant \psi_1$ a.e. on $\Omega$, such that for $i\rightarrow +\infty$,
$$u_{i}\rightarrow u \ \ a.e. \ \ \  in \ \  \Omega$$
and $$u_{i}\rightarrow u \ \ \ in \ \ \ W^{1,1}(\Omega).$$
\end{definition}
Throughout this paper we define
\begin{equation*}
\textbf{I}^{[\psi_1]}_{\beta}(x,R):=\int_{0}^{R}  \frac{D\Psi_1(B_{\rho}(x))}{\rho^{n-\beta }}\frac{\operatorname{d}\!\rho}{\rho},
\end{equation*}
and
\begin{equation*}
\textbf{I}^{[\psi_2]}_{\beta}(x,R):=\int_{0}^{R}  \frac{D\Psi_2(B_{\rho}(x))}{\rho^{n-\beta }}\frac{\operatorname{d}\!\rho}{\rho}
\end{equation*}
with $$D\Psi_1(B_{\rho}(x)):=\int_{B_{\rho}(x)}|\operatorname{div}\left({a}(x,D\psi_1)\right)|d\xi,$$ and $$ D\Psi_2(B_{\rho}(x)):=\int_{B_{\rho}(x)}|\operatorname{div}\left({a}(x,D\psi_2)\right)| d\xi$$
respectively.

Following this, we state our regularity assumptions on $a(\cdot,\cdot)$, we first denote
$$\theta(a,B_{r}(x_0))(x):=\sup_{\eta \in \mathbb{R}^n\setminus \left\lbrace0\right\rbrace  }\frac{|a(x,\eta)-\overline{a}_{B_{r}(x_0)}(\eta)|}{g(|\eta|)}, $$
where $$\overline{a}_{B_{r}(x_0)}(\eta):=\fint_{B_{r}(x_0)}a(x,\eta)dx.$$
Thus, it can be readily confirmed from \eqref{a(x)1} that $|\theta(a,B_{r}(x_0))|\leqslant2L$.
\begin{definition}
We say that  $a(x,\eta)$ is ($\delta$, R)-vanishing for some $\delta, R>0$,  if
\begin{equation}\label{a(x)2}
\omega(R):=\sup_{{\substack{ x_{0}\,\in\,\Omega\\0<r\leq R}}} \left( \fint_{B_{r}(x_{0})}\theta(a,B_{r}(x_{0}))^{\gamma'}\operatorname{d}\!x\right) ^{\frac{1}{\gamma'}} \leq\delta,
\end{equation}
where $\gamma'=\frac{\gamma}{\gamma-1}$, $\gamma$ is as in ~\cite[Theorem 9]{de1}.
\end{definition}

We now present the principal results of this manuscript. The following theorem  establishes the existence of solutions for   double obstacle problems with measure data.
\begin{theorem}
 Under the assumptions  \eqref{a(x)1} and \eqref{a(x)3}, assume that $1+i_g\leqslant n$, $h \in W^{1,G}(\Omega)$ be given boundary data with $\psi_2 \geqslant h\geqslant \psi_1$ a.e. on $\Omega$, and let $u_i \in h+W_{0}^{1,G}(\Omega)$ with $\psi_2 \geqslant u_i \geqslant \psi_1$ solves the variational inequality
\begin{equation}
\int_{\Omega}a(x,Du_i)\cdot D(v-u_i)dx \geqslant \int_{\Omega}f_i(v-u_i)dx
\end{equation}
for all $v \in h+W_{0}^{1,G}(\Omega)$ with $\psi_2 \geqslant  v\geqslant \psi_1$ a.e. in $\Omega$, where $f_i \in L^{1}(\Omega) \cap (W^{1,G}(\Omega))'$ satisfy
$$F:=\sup_{i\in \mathbb{N}}\parallel f_i \parallel_{L^{1}(\Omega)}<+\infty.$$
Then there exists a subsequence $\left\lbrace i_j \right\rbrace \subset \mathbb{N} $ and a limit map $u \in \mathcal{T}_{h}^{1,G}(\Omega)$ with $\psi_2 \geqslant u\geqslant \psi_1$ such that
$u_{i_j}\rightarrow u$ in the sense of Definition \ref{opdy}.
\end{theorem}
\begin{remark}
Our previous study \cite{xiong1} has proven the existence of approximating solutions that converge in the manner described in Definition \ref{opdy} for the single obstacle problem. Subsequently, the existence discussed in this paper can be attained through minor adaptations.
We omit its proof.
\end{remark}
Our second result is the gradient Riesz estimates  for the limits of these approximating solutions to $OP(\psi_1 ; \psi_2; \mu)$.
\begin{theorem}\label{th1}
 Under the assumptions  \eqref{a(x)1},  \eqref{a(x)3} and   \eqref{a(x)2},  assume that $u \in W^{1,1}(\Omega)$ with $\psi_2 \geqslant  u\geqslant \psi_1$ a.e. is a limit of approximating solutions to $OP(\psi_1; \psi_2; \mu)$ with measure data $\mu \in \mathcal{M}_{b}(\Omega)$(in the sense of Definition \ref{opdy}),  and assume that $\omega(\cdot)^{\frac{1}{1+s_g}}$ is Dini-BMO regular, that is
\begin{equation}\label{dytj}
\sup_{r>0}\int_{0}^{r}[\omega(\rho)]^{\frac{1}{1+s_g}}\frac{d\rho}{\rho}< +\infty,
\end{equation}
 Then there exists a constant $c=c(data,\beta,\omega(\cdot))$ such that
 \begin{eqnarray}\label{gdex} \nonumber
&& g(|Du(x_0)|) \\
 &\leqslant& c \left( \mathbf{\RNum{1}}^{|\mu|}_{1}(x_0,2R)+ \mathbf{\RNum{1}}^{[\psi_1]}_{1}(x_0,2R)+ \mathbf{\RNum{1}}^{[\psi_2]}_{1}(x_0,2R)\right) +cg\left(\fint_{B_{R}(x_0)}|Du|dx\right)
 \end{eqnarray}
 where $x_0 \in \Omega$ is the Lebesgue point of $Du$,  $B_{2R}(x_0)\subseteq \Omega$ and $\beta$ is as in Lemma \ref{zcth}.
\end{theorem}
\begin{remark}
To the best of our knowledge, very limited research exists on the gradient estimate associated with double obstacle problems, and our work introduces a new approach, providing a fresh perspective on the solutions to these double obstacle problems.
\end{remark}

 Furthermore, as a consequence of Theorem \ref{th1}, we are able to derive criteria for gradient continuity of solutions to double  obstacle problems.This is expressed in the following
\begin{theorem}\label{th2}
Suppose that the above assumptions of Theorem \ref{th1}   are satisfied, and moreover , if
\begin{equation}\label{limm}
\lim_{R\rightarrow0}\mathbf{\RNum{1}}^{|\mu|}_{1}(\cdot,R)=\lim_{R\rightarrow0}\mathbf{\RNum{1}}^{[\psi_1]}_{1}(\cdot,R)=\lim_{R\rightarrow0}\mathbf{\RNum{1}}^{[\psi_2]}_{1}(\cdot,R)=0 \ \ \ \ \  locally \ uniformly \ in \ \Omega \ with \ respect \ to \ x,
\end{equation}
then $Du$ is continuous in $\Omega$.
\end{theorem}

The remainder of this paper is organized as follows. Section 2 contains some notions and preliminary results. In Section 3,   we  obtain  some comparison estimates.  In Section 4, we complete the proof of several theorems.

\section{Preliminaries}\label{section2}
Throughout this paper,  we shall adopt the convention of denoting by $c$  a constant that may vary from line to line. In order to shorten notation, we collect the dependencies of certain constants on the parameters of our problem as $$data = data(n,i_g,s_g,l,L).$$ Additionally, $A\lesssim B$ means $A\leqslant cB$,  $A\approx B $  means  $A\lesssim B$ and $B \lesssim A$. For an integrable map $f: \Omega \rightarrow \mathbb{R}^n $, we write
$$(f)_{\Omega}:=\fint_{\Omega}fdx:=\frac{1}{|\Omega|}\int_{\Omega}fdx.$$
For $q\in[1,\infty)$, it is easily verified that
\begin{equation} \label{1.8}
\parallel f-(f)_{\Omega}\parallel_{L^q(\Omega)}\leqslant2\min_{c\in \mathbb{R}^m}\parallel f-c\parallel_{L^q(\Omega)}.
\end{equation}
\begin{definition}
A Young function $G$ is called an $N$-function if
$$0<G(t)<+\infty \ \ for \ t>0$$
and
\begin{equation*}\label{nhanshu}
\lim_{t\rightarrow+\infty}\frac{G(t)}{t}=\lim_{t\rightarrow0}\frac{t}{G(t)}=+\infty.
\end{equation*}
It's obvious that $G(t)$ defined as \eqref{g} is an $N$-function.

The Young conjugate  of a Young function $G$ will be denoted by $G^{\ast}$ and defined as
$$G^{\ast}(t)=\sup_{s\geq 0}\left\lbrace st-G(s)\right\rbrace  \ \ for \ t\geq 0.$$
\end{definition}
In particular,  if $G$ is an $N$-function, then $G^{\ast}$ is  an $N$-function as well.
\begin{definition}
A Young function $G$ is said to satisfy the global $\vartriangle_2$ condition, denoted by $G\in\vartriangle_2$, if there exists a positive constant $c$ such that for every $t>0$,
\begin{equation*}
G(2t)\leq cG(t).
\end{equation*}
Similarly, a Young function $G$ is said to satisfy the global $\bigtriangledown_2$ condition, denoted by $G\in\bigtriangledown_2$, if there exists a  constant $\theta >1$ such that for every $t>0$,
\begin{equation*}
G(t)\leq \frac{G(\theta t)}{2\theta}.
\end{equation*}
\end{definition}

\begin{remark}  \label{remark1}
For an increasing function $f: \mathbb{R}^+\rightarrow\mathbb{R}^+$ satisfying  $\vartriangle_2$ condition $f(2t)\lesssim f(t)$ for $t\geqslant0$, it is easy to prove that $f(t+s)\leqslant c[f(t)+f(s)]$ holds for every $t,s\geqslant0$.
\end{remark}

Subsequently, let us revisit a fundamental  property of an
$N$-function, essential for forthcoming developments.
\begin{lemma}\cite{b13}\label{gyoung}
If $G$ is an $N$-function, then $G$ satisfies the following Young's inequality
$$st\leq G^{*}(s)+G(t), \ \ \ for \ \ \forall s,t\geq0.$$
Furthermore, if $G\in \bigtriangleup_{2}\cap \bigtriangledown_{2}$ is an $N$-function, then $G$ satisfies the following Young's inequality with $\forall \varepsilon >0$,
$$st\leq \varepsilon G^{*}(s)+c(\varepsilon)G(t), \ \ \ for \ \ \forall s,t\geq0.$$
\end{lemma}
Note that $G(t)$, defined as \eqref{g}, belongs to $\bigtriangleup_{2}\cap \bigtriangledown_{2}$ and is an $N$-function and therefore  satisfies the Young's inequality. Another important property of Young's conjugate function is the following inequality, which  can be found in \cite{a1}:
\begin{equation}\label{a(x)4}
G^{*}\left( \frac{G(t)}{t}\right) \leqslant G(t).
\end{equation}

Next we define
\begin{equation*}
V_{g}(z):=\left[ \frac{g(|z|)}{|z|}\right] ^{\frac{1}{2}}z,
\end{equation*}
then we have an anlog of  a quantity in the study of the $p-$Laplacian operator,
\begin{equation}\label{vgz1}
|V_{g}(z_1)-V_{g}(z_2)|^{2}\approx \frac{g(|z_1|+|z_2|)}{|z_1|+|z_2|}|z_1-z_2|^{2}\approx g'(|z_1|+|z_2|)|z_1-z_2|^{2}.
\end{equation}
By Lemma 3 in \cite{de1}, we obtain
\begin{equation}\label{vgs}
[a(x,z_1)-a(x,z_2)]\cdot (z_1-z_2)\approx |V_{g}(z_1)-V_{g}(z_2)|^{2}
\end{equation}
Combining the two estimates to get
\begin{align}\label{gvgs}\nonumber
G(|z_1-z_2|)\leq c\frac{g(|z_1-z_2|)}{|z_1-z_2|}|z_1-z_2|^2\leq c\frac{g(|z_1|+|z_2|)}{|z_1|+|z_2|}|z_1-z_2|^{2}\\
\leq c[a(x,z_1)-a(x,z_2)]\cdot (z_1-z_2)
\end{align}

In preparation for proving our forthcoming results, it is essential to elucidate certain aspects regarding the functions $g$ and $G$ and the embedding relationships between the Orlicz and Lebesgue spaces. To this end, we recall the following lemma, with its proof provided in \cite[Lemma 3.1]{xiong2}.
\begin{lemma}\label{ag}
Assume that $g(t)$ satisfies   \eqref{a(x)3}, $G(t)$ is defined in \eqref{g}. Then we have

(1) for any  $\beta \geq1$,

 \ \ \ \ \ \ \ \ \ $\beta^{i_g}\leq \dfrac{g(\beta t)}{g(t)}\leq \beta^{s_g}$ \ \ \ and \ \ \  $\beta^{1+i_g}\leq \dfrac{G(\beta t)}{G(t)} \leq \beta^{1+s_g}$,  \ \ \  for every $t>0$,

for any  $0<\beta<1$,

 \ \ \ \ \ \ \ \ \ $\beta^{s_g}\leq \dfrac{g(\beta t)}{g(t)}\leq \beta^{i_g}$ \ \ \ and \ \ \
$\beta^{1+s_g}\leq \dfrac{G(\beta t)}{G(t)} \leq \beta^{1+i_g}$, \ \ \ for every $t>0$.

(2) for any  $\beta \geq1$,

 \ \ \ \ \ \ \ \ \ $\beta^{\frac{1}{s_g}}\leq \dfrac{g^{-1}(\beta t)}{g^{-1}(t)}\leq \beta^{\frac{1}{i_g}}$ \ \ \ and \ \ \  $\beta^{\frac{1}{1+s_g}}\leq \dfrac{G^{-1}(\beta t)}{G^{-1}(t)} \leq \beta^{\frac{1}{1+i_g}}$,  \ \ \  for every $t>0$,

for any  $0<\beta<1$,

 \ \ \ \ \ \ \ \ \ $\beta^{\frac{1}{i_g}}\leq \dfrac{g^{-1}(\beta t)}{g^{-1}(t)}\leq \beta^{\frac{1}{s_g}}$ \ \ \ and \ \ \
$\beta^{\frac{1}{1+i_g}}\leq \dfrac{G^{-1}(\beta t)}{G^{-1}(t)} \leq \beta^{\frac{1}{1+s_g}}$, \ \ \ for every $t>0$.
\end{lemma}
It's  apparent that lemma \ref{ag} indicates that
\begin{equation}\label{lg}
L^{1+s_g}(\Omega)\subset L^{G}(\Omega) \subset L^{1+i_g}(\Omega) \subset L^1(\Omega)
\end{equation}
and $g(\cdot), g^{-1}(\cdot), G(\cdot), G^{-1}(\cdot)$ satisfy the global $\bigtriangleup_2$ condition.

Following this, we introduce a Sobolev-type embedding for the function $g$.
\begin{lemma}(see \cite{b13}, Proposition 3.4) \label{gudu}
Assume that $B_R(x_0) \subseteq \Omega$, and $g: [0, +\infty)\rightarrow [0, +\infty)$ is a positive increasing  function satisfying \eqref{a(x)3}. Then there exists a constant $c=c(n,i_g,s_g)$ such that
\begin{equation*}
\fint_{B_R}\left[ g\left(\frac{|u|}{R} \right) \right] ^{\frac{n}{n-1}}dx \leqslant c \left( \fint_{B_R}g(|Du|)dx\right) ^{\frac{n}{n-1}}
\end{equation*}
for every weakly differentiable function $u \in W_{0}^{1,g}(B_R(x_0))$.
\end{lemma}

The subsequent lemma presents  Lipschitz regularity and excess decay estimates for homogeneous equations with constant coefficients.
\begin{lemma}(see \cite{b13}, Lemma 4.1)\label{zcth}
If $ w\in W_{loc}^{1,G}(\Omega)$ is a local weak solution of
\begin{equation*}
-\operatorname{div}\left(  a(Dw)\right) =0 \ \ \ \ \ \ in\ \ \ \Omega,
\end{equation*}
 where    $a(x,\eta)=a(\eta)$ satisfies  the assumptions \eqref{a(x)1} and \eqref{a(x)3}. For every ball $B_{R}(x_0)\subseteq \Omega$ ,then we have the following $De \ Giorgi \  type$ estimate:
\begin{equation*}
\sup_{B_{\frac{R}{4}}(x_0)}|Dw| \leqslant c_1\fint_{B_R(x_0)}|Dw|dx.
\end{equation*}
Moreover,  there exist constant $\beta \in (0,1)$ such that

\begin{equation*}
\fint_{B_\rho(x_0)} \vert Dw-(Dw)_{B_\rho(x_0)}\vert \operatorname{d}\!\xi\leq c_2 \left( \frac{\rho}{R}\right) ^\beta \fint_{B_R(x_0)} \vert Dw-(Dw)_{B_R(x_0)}\vert \operatorname{d}\!\xi,
\end{equation*}
\begin{equation*}
|Dw(x_1)-Dw(x_2)| \leqslant c_3 \left( \frac{\rho}{R}\right) ^\beta \fint_{B_R(x_0)} \vert Dw \vert \operatorname{d}\!\xi
\end{equation*}
where $0<\rho\leqslant R$, $ x_1, x_2 \in B_{\frac{\rho}{2}}(x_0).$ The exponent $\beta$ and the constants $c_1,c_2,c_3$ share the same dependence on $data$.
\end{lemma}

\section{Comparison estimates and regularity results }\label{section3}

In this section we want to obtain some comparison estimates between   the solutions to double obstacle problems and  to homogeneous elliptic equations.  Hence, a corresponding excess decay estimate can be achieved for solutions of double obstacle problems with measure data. Primarily, we will demonstrate a comparison estimate between solutions of double obstacle problems with measure data and those of single obstacle problems.

We introduce three functions that are directly dependent on  $g$:
\begin{equation*}
f_{\chi}(t):=\int_{0}^{t}\left[ \frac{g(s)}{s}\right]^{1+\chi}ds, \ \ \ g_{\chi}(t):=\left[ \frac{g(t)}{t}\right]^{1+\chi}t, \ \ \ h_{\chi}(t):=\frac{g_{\chi}(t)}{t},
\end{equation*}
for $\chi\geqslant-1$. It's obvious that $g_{\chi}(\cdot)$ and $h_{\chi}(\cdot)$ are increasing and satisfy $\bigtriangleup_2$ condition, therefore by Remark \ref{remark1} to get $$g_{\chi}(t+s)\leqslant c[g_{\chi}(t)+g_{\chi}(s)], \ \ \ h_{\chi}(t+s)\leqslant c[h_{\chi}(t)+h_{\chi}(s)].$$

\begin{lemma}\label{dudw1-}
Assume that conditions  \eqref{a(x)1}-\eqref{a(x)3} are fulfilled, let $B_{2R}(x_0)\subset \Omega,  f \in L^{1}(B_R(x_0))\cap (W^{1,G}(B_R(x_0)))'$ and the map $u\in W^{1,G}(B_R(x_0))$ with $\psi_2 \geqslant u \geqslant \psi_1$ solves the variational inequality
\begin{equation}\label{bju-}
\int_{B_R(x_0)} a(x,Du)\cdot D(v-u)dx \geq \int_{B_R(x_0)} f(v-u) dx
\end{equation}
for any  $v \in u+W_0^{1,G}(B_R(x_0))$ that  satisfy $\psi_2 \geqslant  v \geqslant \psi_1 $ a.e.  in $B_R(x_0)$.
Let $w_0 \in u+W_{0}^{1,G}(B_R(x_0))$ with $w_0 \geq \psi_1 $ be the weak solution of the single obstacle problem
\begin{equation}\label{bjw1-}
\int_{B_R(x_0)} a(x,Dw_0)\cdot D(v-w_0)dx \geqslant \int_{B_R(x_0)} a(x,D\psi_2)\cdot D(v-w_0)dx
\end{equation}
for any  $v \in w_0+W_0^{1,G}(B_R(x_0))$ that  satisfy $ v \geqslant \psi_1 $ a.e.  in $B_R(x_0)$.
Then we obtain
\begin{equation}\label{ga0-}
\fint_{B_{R}(x_0)}g_{\chi}(|Du-Dw_0|)dx\leqslant c_{1}g_{\chi}(A_0),
\end{equation}
\begin{equation}\label{ha0-}
\fint_{B_{R}(x_0)}h_{\chi}(|Du-Dw_0|)dx\leqslant c_{1}h_{\chi}(A_0),
\end{equation}
\begin{equation}\label{ga01-}
\fint_{B_R(x_0)}[g(|Du-Dw_0|)]^{\xi}dx\leqslant c_{2}\left[ R\fint_{B_R(x_0)}|f|dx+\frac{D\Psi_2(B_R(x_0))}{R^{n-1}} \right] ^{\xi}
\end{equation}
for
$$A_0:=g^{-1}\left(R\fint_{B_R(x_0)}|f|dx+\frac{D\Psi_2(B_R(x_0))}{R^{n-1}} \right),$$
$$ \chi \in \left[-1, \min \left\lbrace \frac{1}{s_g-1}, \frac{s_g}{(s_g-1)(n-1)}\right\rbrace   \right),  \xi \in \left[ 1, \min \left\lbrace \frac{s_g+1}{s_g}, \frac{n}{n-1}\right\rbrace  \right)   $$ and with constants $c_1=c_1(data,\chi),  c_2=c_2(data,\xi)$.
\end{lemma}

\begin{proof}
Since $w_0 \in u+W_0^{1,G}(B_R(x_0)), u \leqslant \psi_2$ a.e. in $B_R(x_0)$, we consequently deduce $(w_0-\psi_2)_+ \in W_0^{1,G}(B_R(x_0))$. Subsequently, we choose $v=\min \{w_0,\psi_2\}=w_0-(w_0-\psi_2)_+ \in w_0+ W_0^{1,G}(B_R(x_0))$ with $v \geqslant \psi_1$ as comparison functions in \eqref{bjw1-},   and it can be inferred from \eqref{gvgs} that
\begin{eqnarray*}
\int_{B_R(x_0)}G(|D(w_0-\psi_2)_+|)dx &\leqslant& c \int_{B_R(x_0)}[a(x,Dw_0)-a(x,D\psi_2)] \cdot D[(w_0-\psi_2)_+]dx \\
&\leqslant& 0.
\end{eqnarray*}
In view of the fact that $G(\cdot)$ is increasing over $[0,+\infty)$ and $G(0)=0$,  we can conclude that
$$D(w_0-\psi_2)_+=0 \ \ a.e. \ \ \ in \ \ B_R(x_0).$$
Together with  $(w_0-\psi_2)_+=0\ \ \  on \ \ \partial B_R(x_0)$, which implies
$$(w_0-\psi_2)_+=0 \ \ \ a.e.\ \ in \ \ B_R(x_0).$$
This indicates that
$w_0 \leqslant \psi_2$ a.e. in $B_R(x_0)$.

Without loss of generality we may assume that $A_0>0$, otherwise, by \eqref{vgz1} and \eqref{vgs} to get $u=w_0$ in $B_R(x_0)$. So we define
$$\overline{u}(x)=\frac{u(x_0+Rx)}{A_0R}, \ \ \ \overline{w_0}(x)=\frac{w_0(x_0+Rx)}{A_0R}, \ \ \ \overline{a}(x,z)=\frac{a(x_0+Rx,A_0z)}{g(A_0)},$$
$$\overline{g}(x)=\frac{g(A_0x)}{g(A_0)}, \ \ \ \overline{f}(x)=R\frac{f(x_0+Rx)}{g(A_0)},  \ \ \ \overline{G}(t)=\int_{0}^{t}\tau\overline{a}(\tau)d\tau,$$
$$\overline{\psi_1}(x)=\frac{\psi_1(x_0+Rx)}{A_0R}, \ \  \overline{\psi_2}(x)=\frac{\psi_2(x_0+Rx)}{A_0R}, \ \ D\overline{\Psi_2}(B_1):=\int_{B_1}|\operatorname{div}\left(\overline{a}(x,D\overline{\psi_2})\right) |dx,$$
Consequently, after appropriate rescaling,  we deduce  $\overline{u}\in W^{1,G}(B_1)$ with $\overline{\psi_2} \geqslant \overline{u} \geqslant\overline{\psi_1}$ solves the variational inequality
\begin{equation}\label{bar1-}
\int_{B_1} \overline{a}(x,D\overline{u})\cdot D(\overline{v}-\overline{u})dx \geq \int_{B_1} \overline{f}(\overline{v}-\overline{u}) dx
\end{equation}
for any  $\overline{v}\in \overline{u}+W_0^{1,G}(B_1)$ that  satisfy $\overline{\psi_2} \geqslant  \overline{v} \geqslant \overline{\psi_1}$ a.e.  in $B_1$.
And $\overline{w}_0 \in \overline{u}+W_{0}^{1,G}(B_1)$ with $\overline{w}_0 \geq \overline{\psi_1}$ solves the variational inequality
\begin{equation}\label{bar2-}
\int_{B_1} \overline{a}(x,D\overline{w}_0)\cdot D(\overline{v}-\overline{w}_0)dx \geqslant \int_{B_1} \overline{a}(x,D\overline{\psi_2})\cdot D(\overline{v}-\overline{w}_0)dx
\end{equation}
for any  $\overline{v}\in \overline{w_0}+W_0^{1,G}(B_1)$ that  satisfy $\overline{v}\geq \overline{\psi_1}$ a.e.  in $B_1$.
Moreover, through a series of calculations, we can deduce
$$D_{\eta} \overline{a}(x,\eta )\lambda \cdot \lambda \geqslant l\dfrac{\overline{g}(|\eta|)}{|\eta|}|\lambda|^2, \ \ \ \ \ i_g\leqslant \frac{t\overline{g}'(t)}{\overline{g}(t)}\leqslant s_g,$$
\begin{eqnarray*}
\int_{B_1}|\overline{f}|dx+D\overline{\Psi_2}(B_1)=1.
\end{eqnarray*}
Following this, we investigate two cases, starting with the case of slow growth:
$$\int^{\infty}\left( \frac{s}{\overline{G}(s)}\right)^{\frac{1}{n-1}}ds=\infty.$$
We define
\begin{equation*}
F_{\chi}(t):=
\left\{\begin{array}{r@{\ \ }c@{\ \ }ll}
0\ \ \ \ \ \ \ \ \  \ \  if\ \ t=0\,, \\[0.05cm]
\overline{f_{\chi}}(1)t\ \ \ \ \ \ \ if\ \ t\in(0,1)\,, \\[0.05cm]
\overline{f_{\chi}}(t)\ \ \ \ \ \ \ if\ \ t\in[1,\infty)\,, \\[0.05cm]
\end{array}\right.
\end{equation*}
$$\overline{f_{\chi}}(t):=\int_{0}^{t}\left[ \frac{\overline{g}(s)}{s}\right]^{1+\chi}ds,
\ \ \  \Phi_k(t):=T_1(t-T_k(t)),
$$
$$\mathcal{F}:=\left( \int_{B_1}F_{\chi}(|D\overline{u}-D\overline{w_0}|)dx\right)^{\frac{1}{n}},  $$
where $T_k(t)$ is as in \eqref{tks}. Now we take $$\overline{v_1}=\overline{u}+T_k\left( \frac{\overline{w_0}-\overline{u}}{c_n\mathcal{F}}\right)c_n\mathcal{F}$$ and $$\overline{v_2}=\overline{w_0}+T_k\left( \frac{\overline{u}-\overline{w_0}}{c_n\mathcal{F}}\right)c_n\mathcal{F},$$
which satisfy $\overline{\psi_2} \geqslant \overline{v_1} \geqslant  \overline{\psi_1}$ and $\overline{v_2} \geqslant  \overline{\psi_1}$ a.e. in $B_1$,
as comparison functions in the inequalities \eqref{bar1-} and \eqref{bar2-} separately, then by \eqref{gvgs}  to get
\begin{eqnarray*}
\int_{C_k}\overline{G}(|D\overline{u}-D\overline{w_0}|)dx &\leqslant& c \int_{B_1}[\overline{a}(x,D\overline{u})-\overline{a}(x,D\overline{w_0})]\cdot DT_k\left(\frac{\overline{u}-\overline{w_0}}{c_n \mathcal{F}} \right)c_n\mathcal{F} dx \\
&\leqslant &c \int_{B_1} \left[ \overline{|f|} +|\operatorname{div}\overline{a}(x,D\overline{\psi_2})|\right] T_k\left(\frac{\overline{u}-\overline{w_0}}{c_n \mathcal{F}} \right)c_n\mathcal{F} dx \\
&\leqslant & ck\mathcal{F}\left[ \int_{B_1}|\overline{f}|dx+D\overline{\Psi_2}(B_1)\right] \\
&\leqslant & ck\mathcal{F},
\end{eqnarray*}
where $$C_k:= \left\lbrace x\in B_1: \frac{|\overline{u}-\overline{w_0}|}{c_n \mathcal{F}} \leqslant k \right\rbrace. $$
In a similar manner, we choose $$\overline{v_1}=\overline{u}+\Phi_k\left( \frac{\overline{w_0}-\overline{u}}{c_n\mathcal{F}}\right)c_n\mathcal{F}$$ and $$\overline{v_2}=\overline{w_0}+\Phi_k\left( \frac{\overline{u}-\overline{w_0}}{c_n\mathcal{F}}\right)c_n\mathcal{F},$$
which satisfy $\overline{\psi_2} \geqslant \overline{v_1} \geqslant  \overline{\psi_1}$ and $\overline{v_2} \geqslant  \overline{\psi_1}$ a.e. in $B_1$,  as test functions, then we infer
\begin{equation*}
\int_{D_k}\overline{G}(|D\overline{u}-D\overline{w_0}|)dx\leqslant c\mathcal{F},
\end{equation*}
where
$$D_k:= \left\lbrace x\in B_1: k<\frac{|\overline{u}-\overline{w_0}|}{c_n \mathcal{F}} \leqslant k+1 \right\rbrace. $$
Therefore, we obtain (see Step 2.1 of Lemma 5.1 in \cite{b13})
\begin{equation*}
\int_{B_1}\overline{g}_{\chi}(|D\overline{u}-D\overline{w_0}|)dx\leqslant c,
\end{equation*}
where $\overline{g}_{\chi}(t):=\left[ \frac{\overline{g}(t)}{t}\right]^{1+\chi}t.$

For the fast growth case:
$$\int^{\infty}\left( \frac{s}{\overline{G}(s)}\right)^{\frac{1}{n-1}}ds<\infty,$$
we take $\overline{v_1}=\overline{u}+\frac{\overline{w_0}-\overline{u}}{2}\geqslant \overline{\psi_1}$ and $\overline{v_2}=\overline{w_0}+\frac{\overline{u}-\overline{w_0}}{2}\geqslant \overline{\psi_1}$, which satisfy $\overline{\psi_2} \geqslant \overline{v_1} \geqslant  \overline{\psi_1}$ and $\overline{v_2} \geqslant  \overline{\psi_1}$ a.e. in $B_1$,  as comparison functions in the inequalities \eqref{bar1-} and \eqref{bar2-} separately, then by Sobolev's embedding (see Proposition 3.3 in \cite{b13}) to get
\begin{eqnarray*}
\int_{B_1}\overline{G}(|D\overline{u}-D\overline{w_0}|)dx&\leqslant& c\int_{B_1}\left[ \overline{|f|} +|\operatorname{div}\overline{a}(x,D\overline{\psi_2})|\right] |\overline{u}-\overline{w_0}|dx \\
&\leqslant& c||D\overline{u}-D\overline{w_0}||_{L^{G}(B_1)}.
\end{eqnarray*}
Consequently, we derive (see Step 2.2 of Lemma 5.1 in \cite{b13})
\begin{equation*}
\int_{B_1}\overline{g}_{\chi}(|D\overline{u}-D\overline{w_0}|)dx\leqslant c,
\end{equation*}
and we obtain \eqref{ga0-}.
Likewise, we derive \eqref{ha0-} and \eqref{ga01-}; for a  detailed proofing process, refer to Corollary 5.2 and Lemma 5.3 in \cite{b13}.
\end{proof}

\begin{corollary}\label{dudw1c}
Assume that conditions  \eqref{a(x)1}-\eqref{a(x)3} are fulfilled, let $w_0$ be as in Lemma \ref{dudw1-} and $\mu \in  \mathcal{M}_{b}(\Omega)$ and u be a limit of approximating solutions for $OP(\psi_1; \psi_2; \mu)$, in the sense of Definition \ref{opdy}. Then we derive
\begin{equation*}
\fint_{B_{R}(x_0)}g_{\chi}(|Du-Dw_0|)dx\leqslant c_{1}\left[ g_{\chi}(A_1)+g_{\chi}(A_3)\right],
\end{equation*}
\begin{equation*}
\fint_{B_{R}(x_0)}h_{\chi}(|Du-Dw_0|)dx\leqslant c_{1}\left[ h_{\chi}(A_1)+h_{\chi}(A_3)\right],
\end{equation*}
\begin{equation*}
\fint_{B_R(x_0)}[g(|Du-Dw_0|)]^{\xi}dx\leqslant c_{2}\left[ \frac{|\mu|(\overline{B_{R}(x_0)})}{R^{n-1}}+\frac{D\Psi_2(B_{R}(x_0))}{R^{n-1}}\right] ^{\xi}
\end{equation*}
for
$$A_1:=g^{-1}\left( \frac{|\mu|(\overline{B_R(x_0)})}{R^{n-1}}\right), \ \ \ \ A_3:=g^{-1}\left( \frac{D\Psi_2(B_R(x_0))}{R^{n-1}}\right)$$
and $ \chi, \xi,c_1,c_2$ are as in Lemma \ref{dudw1-}.
\end{corollary}
\begin{proof}
By Definition \ref{opdy}, there exists functions
$$f_{i} \in (W^{1,G}(B_R(x_0)))'\cap L^{1}(B_R(x_0))\ \  with\ \  f_{i}\stackrel{\ast}\rightharpoonup \mu \ in \ \mathcal{M}_{b}(B_R(x_0)) \ \ as \ i\rightarrow+\infty$$
 satisfies
$$\limsup_{i\rightarrow+\infty}\int_{B_R(x_0)}|f_i|dx\leqslant|\mu|(\overline{B_R(x_0)}).$$
and solutions $u_i\in W^{1,G}(B_R(x_0))$ of the obstacle problems \eqref{opdy1} with
$$u_i\rightarrow u \ \ a.e. \ \  in \ \ B_R(x_0) $$ and $$u_i\rightarrow u \  \ in  \ \ W^{1,1}(B_R(x_0)).$$
Thus, utilizing F.Riesz's theorem and Fatou's lemma, we finalize the proof.
\end{proof}

The lemma presented below establishes comparison estimates between inhomogeneous obstacle problems and homogeneous obstacle problems.
\begin{lemma}\label{dudw1}
Assume that conditions  \eqref{a(x)1}-\eqref{a(x)3} are fulfilled, let $B_{2R}(x_0)\subset \Omega$ and the map $w_0\in W^{1,G}(B_R(x_0))$ with $w_0 \geq \psi_1$ solves the variational inequality \eqref{bjw1-}
Let $w_1 \in w_0+W_{0}^{1,G}(B_R(x_0))$ with $w_1 \geq \psi_1 $ be the weak solution of the homogeneous obstacle problem
\begin{equation}\label{bjw1}
\int_{B_R(x_0)} a(x,Dw_1)\cdot D(v-w_1)dx \geq 0
\end{equation}
for any  $v \in w_1+W_0^{1,G}(B_R(x_0))$ that  satisfy $ v \geqslant \psi_1 $ a.e.  in $B_R(x_0)$.
Then we have
\begin{equation}\label{ga0}
\fint_{B_{R}(x_0)}g_{\chi}(|Dw_0-Dw_1|)dx\leqslant c_{1}g_{\chi}(A_3),
\end{equation}
\begin{equation}\label{ha0}
\fint_{B_{R}(x_0)}h_{\chi}(|Dw_0-Dw_1|)dx\leqslant c_{1}h_{\chi}(A_3),
\end{equation}
\begin{equation}\label{ga01}
\fint_{B_R(x_0)}[g(|Dw_0-Dw_1|)]^{\xi}dx\leqslant c_{2}\left[ \frac{D\Psi_2(B_{R}(x_0))}{R^{n-1}}\right] ^{\xi}
\end{equation}
for
$A_3, \chi,  \xi, c_1, c_2$ are as in Lemma \ref{dudw1-} and Corollary \ref{dudw1c}.
\end{lemma}

\begin{proof}
Without loss of generality we may assume that $A_3>0$, then we define
$$\overline{w_0}(x)=\frac{u(x_0+Rx)}{A_3R}, \ \ \ \overline{w_1}(x)=\frac{w_1(x_0+Rx)}{A_3R}, \ \ \ \overline{a}(x,z)=\frac{a(x_0+Rx,A_3z)}{g(A_3)},$$
$$\overline{g}(x)=\frac{g(A_3x)}{g(A_3)}, \ \ \  \overline{\psi_1}(x)=\frac{\psi_1(x_0+Rx)}{A_3R}, \ \     \overline{\psi_2}(x)=\frac{\psi_2(x_0+Rx)}{A_3R}, \ \ \ $$
 $$D\overline{\Psi_2}(B_1):=\int_{B_1}|\operatorname{div}\left(\overline{a}(x,D\overline{\psi_2})\right)| dx=1.$$
Subsequently, the proof follows a similar structure to that of Lemma \ref{dudw1-}.
For the slow growth case, we take $$\overline{v_1}=\overline{w_0}+T_k\left( \frac{\overline{w_1}-\overline{w_0}}{c_n\mathcal{F}}\right)c_n\mathcal{F}\geqslant  \overline{\psi_1}$$ and $$\overline{v_2}=\overline{w_1}+T_k\left( \frac{\overline{w_0}-\overline{w_1}}{c_n\mathcal{F}}\right)c_n\mathcal{F} \geqslant  \overline{\psi_1}$$ as comparison functions in the inequalities \eqref{bjw1-} and \eqref{bjw1}.

For the fast growth case:
$$\int^{\infty}\left( \frac{s}{\overline{G}(s)}\right)^{\frac{1}{n-1}}ds<\infty.$$
We take $\overline{v_1}=\overline{w_0}+\frac{\overline{w_1}-\overline{w_0}}{2}\geqslant \overline{\psi_1}$ and $\overline{v_2}=\overline{w_1}+\frac{\overline{w_0}-\overline{w_1}}{2}\geqslant \overline{\psi_1}$ as comparison functions in the inequalities \eqref{bjw1-} and \eqref{bjw1},
and therefore we obtain
\begin{eqnarray*}
\fint_{B_{R}(x_0)}g_{\chi}(|Dw_0-Dw_1|)dx &\leqslant& g_{\chi} \left(g^{-1} \left( \frac{\int_{B_R(x_0)}|\operatorname{div}(a(x,D\psi_2))|dx}{R^{n-1}}\right) \right) \\
&\leqslant& c_{1}g_{\chi}(A_3).
\end{eqnarray*}
Similarly, we obtain \eqref{ha0} and \eqref{ga01}.
\end{proof}

Following this, we show a comparison estimate between solutions of a homogeneous obstacle problem and a suitable elliptic equation.
\begin{lemma}\label{dw1w2}
Under the assumptions \eqref{a(x)1}-\eqref{a(x)3}, we assume that $B_{2R}(x_0)\subseteq \Omega$,$w_1 \in W^{1,G}(B_R(x_0))$ with $w_1\geqslant \psi_1$ solves the inequality \eqref{bjw1}. Let $w_2 \in W^{1,G}(B_R(x_0))$  be a weak solution of the equation
\begin{equation}\label{bjw2}
\left\{\begin{array}{r@{\ \ }c@{\ \ }ll}
-\operatorname{div}\left(  a(x, Dw_2)\right)&=&  -\operatorname{div}\left(  a(x, D\psi_1)\right) \ \ \ \  in \ \ B_R(x_0) \,, \\[0.05cm]
w_2&=&w_1  \ \ \ \ \ \ \ \ \ \ \ \ \ \ \ \ \ \ \ \  \ \mbox{on}\ \ \partial B_{R}(x_0) \,. \\[0.05cm]
\end{array} \right.
\end{equation}

Then we obtain
\begin{equation}\label{w1w21}
\fint_{B_{R}(x_0)}g_{\chi}(|Dw_1-Dw_2|)dx\leqslant c_{1}g_{\chi}(A_2),
\end{equation}
\begin{equation}\label{w1w22}
\fint_{B_{R}(x_0)}h_{\chi}(|Dw_1-Dw_2|)dx\leqslant c_{1}h_{\chi}(A_2),
\end{equation}
\begin{equation}\label{w1w23}
\fint_{B_R(x_0)}[g(|Dw_1-Dw_2|)]^{\xi}dx\leqslant c_{2}\left[ \frac{D\Psi_1(B_{R}(x_0))}{R^{n-1}}\right] ^{\xi},
\end{equation}
for
$$A_2:=g^{-1}\left( \frac{D\Psi_1(B_R(x_0))}{R^{n-1}}\right)$$and $\chi,\xi ,c_1, c_2$ are as in Lemma \ref{dudw1-}.
\end{lemma}

\begin{proof}
We test the inequality \eqref{bjw2} with $(\psi_1-w_2)_{+} \in W_0^{1,G}(B_R(x_0))$, then it follows from  \eqref{gvgs} that
\begin{eqnarray*}
\int_{B_R(x_0)}G(|D(\psi_1-w_2)_+|)dx &\leqslant& c \int_{B_R(x_0)}[a(x,D\psi_1)-a(x,Dw_2)] \cdot D[(\psi_1-w_2)_+]dx \\
&\leqslant& 0.
\end{eqnarray*}
Because $G(\cdot)$ is increasing over $[0,+\infty)$ and $G(0)=0$, we can infer that
$$D(\psi_1-w_2)_+=0 \ \ a.e. \ \ \ in \ \ B_R(x_0).$$
Combining with $(\psi_1-w_2)_+=0\ \ \  on \ \ \partial B_R(x_0)$, which implies
$$(\psi_1-w_2)_+=0 \ \ \ a.e.\ \ in \ \ B_R(x_0).$$
It means that
$w_2 \geqslant \psi_1$ a.e. in $B_R(x_0)$.

Then we define
$$\overline{w_1}(x)=\frac{w_1(x_0+Rx)}{A_2R}, \ \ \ \overline{w_2}(x)=\frac{w_2(x_0+Rx)}{A_2R}, \ \ \ \overline{a}(x,z)=\frac{a(x_0+Rx,A_2z)}{g(A_2)},$$
$$\overline{g}(x)=\frac{g(A_2x)}{g(A_2)},  \ \ \ \ \overline{\psi_1}(x)=\frac{\psi_1(x_0+Rx)}{A_2R}, \ \ \ D\overline{\Psi_1}(B_1):=\int_{B_1}|\operatorname{div}\left(\overline{a}(x,D\overline{\psi_1})\right)| dx=1.$$
Subsequently, the proof follows a similar structure to that of Lemma \ref{dudw1-}.

For the slow growth case, we take $$\overline{v}=\overline{w_1}+T_k\left( \frac{\overline{w_2}-\overline{w_1}}{c_n\mathcal{F}}\right)c_n\mathcal{F}\geqslant  \overline{\psi_1}$$ and $$\overline{\varphi}=T_k\left( \frac{\overline{w_1}-\overline{w_2}}{c_n\mathcal{F}}\right)c_n\mathcal{F}$$ as test functions in the inequalities \eqref{bjw1} and equation \eqref{bjw2}.

For the fast growth case,
We take $\overline{v}=\overline{w_1}+\frac{\overline{w_2}-\overline{w_1}}{2} \geqslant \overline{\psi_1}$ and $\overline{\varphi}=\frac{\overline{w_1}-\overline{w_2}}{2}$ as test functions in the inequalities \eqref{bjw1} and \eqref{bjw2}.
Consequently,  we have
\begin{eqnarray*}
\fint_{B_{R}(x_0)}g_{\chi}(|Dw_1-Dw_2|)dx &\leqslant& g_{\chi} \left(g^{-1} \left( \frac{\int_{B_R(x_0)}|\operatorname{div}(a(x,D\psi_1))|dx}{R^{n-1}}\right) \right) \\
&\leqslant& c_{1}g_{\chi}(A_2).
\end{eqnarray*}
Moreover, \eqref{w1w22} and \eqref{w1w23} also hold.
\end{proof}

Based on the findings of Lemma 5.1, Corollary 5.2, and Lemma 5.3 as detailed in \cite{b13}, the ensuing lemma is established.
\begin{lemma}\label{dw2w3}
Under the assumptions \eqref{a(x)1}-\eqref{a(x)3}, we assume that $B_{2R}(x_0)\subset \Omega$, $w_2 \in W^{1,G}(B_R(x_0))$  solves the equation \eqref{bjw2}. Let $w_3 \in W^{1,G}(B_R(x_0))$  be a weak solution of the equation
\begin{equation} \label{bjw3}
\left\{\begin{array}{r@{\ \ }c@{\ \ }ll}
-\operatorname{div}\left(  a(x, Dw_3)\right)&=&0 \ \ \ \  in \ \ B_R(x_0) \,, \\[0.05cm]
w_3&=&w_2  \ \ \  \mbox{on}\ \ \partial B_{R}(x_0) \,. \\[0.05cm]
\end{array} \right.
\end{equation}
Then we have
\begin{equation*}
\fint_{B_{R}(x_0)}g_{\chi}(|Dw_2-Dw_3|)dx\leqslant c_{1}g_{\chi}(A_2),
\end{equation*}
\begin{equation*}
\fint_{B_{R}(x_0)}h_{\chi}(|Dw_2-Dw_3|)dx\leqslant c_{1}h_{\chi}(A_2),
\end{equation*}
\begin{equation*}
\fint_{B_R(x_0)}[g(|Dw_2-Dw_3|)]^{\xi}dx\leqslant c_{2}\left[ \frac{D\Psi_1(B_{R}(x_0))}{R^{n-1}}\right] ^{\xi},
\end{equation*}
for
$A_2,\chi,\xi ,c_1, c_2$ are as in Lemma \ref{dw1w2} and Lemma \ref{dudw1-}.
\end{lemma}
In the sequel, we  present  a weighted type energy estimate.
\begin{lemma}\label{bjuw1-}
Under the hypothesis of Lemma \ref{dudw1-}, then there exists a constant $c=c(data)$ such that
\begin{equation*}
\fint_{B_{R}(x_0)}\frac{|V_{g}(Du)-V_{g}(Dw_0)|^{2}}{(\alpha+|u-w_0|)^{\xi}}dx\leqslant c\frac{\alpha^{1-\xi}}{\xi-1}\left[ \fint_{B_{R}(x_0)}|f|dx+D\Psi_2(B_{R}(x_0)) \right]
\end{equation*}
for $\alpha>0$ and $\xi>1$.
\end{lemma}

\begin{proof}
We consider
$$\eta_{\pm}:=\frac{1}{\xi-1}\left[1-\left(1-\frac{(u-w_0)_{\pm}}{\alpha}\right)^{1-\xi}\right],$$
then $\eta_{\pm} \in W_{0}^{1,G}(B_{R}(x_0))\cap L^{\infty}(B_{R}(x_0))$ and $\eta_{\pm} \geq0$. The function $\eta_{\pm}$ is taken with reference to Lemma 5.1 in \cite{byun1}. Moreover, through a series of calculations, we have
$$u-\alpha\eta_+ \geqslant \min \left\lbrace u,w_0\right\rbrace \geqslant \psi_1,$$
$$u+\alpha\eta_- \leqslant \max \left\lbrace u,w_0\right\rbrace \leqslant \psi_2,$$
$$w_{0}-\alpha\eta_- \geqslant \min \left\lbrace u,w_0\right\rbrace \geqslant \psi_1.$$
Now we choose $v=u\pm \alpha \eta_{\mp}$ and $\bar{v}=w_0 \pm \alpha \eta_{\pm}$, which satisfy $\psi_2 \geqslant v \geqslant \psi_1, \ \overline{v}\geqslant \psi_1$ a.e. in $B_{R}(x_0)$, as comparison functions in the variational inequalities \eqref{bju-} and \eqref{bjw1-} respectively, then by \eqref{vgs} we obtain
\begin{eqnarray*}
&&\int_{B_{R}(x_0)\cap \{ u\geq w_0\}}\frac{|V_{g}(Du)-V_{g}(Dw_0)|^{2}}{(\alpha+|u-w_0|)^{\xi}}dx \\
&\approx& \int_{B_{R}(x_0)\cap \{ u\geq w_0\}} \frac{[a(x,Du)-a(x,Dw_0)]\cdot(Du-Dw_0)}{(\alpha+|u-w_0|)^{\xi}}dx \\
&\leqslant& c\int_{B_{R}(x_0)}\alpha^{1-\xi}\eta_+\left[|f|+|\operatorname{div}(a(x,D\psi_2))|\right]dx \\
&\leqslant& c\frac{\alpha^{1-\xi}}{\xi-1}\left[ \int_{B_{R}(x_0)}|f|dx+D\Psi_2(B_{R}(x_0)) \right].
\end{eqnarray*}
and
\begin{eqnarray*}
&&\int_{B_{R}(x_0)\cap \{ u< w_0\}}\frac{|V_{g}(Du)-V_{g}(Dw_0)|^{2}}{(\alpha+|u-w_0|)^{\xi}}dx \\
&\leqslant& c\int_{B_{R}(x_0)}\alpha^{1-\xi}\eta_-\left[|f|+|\operatorname{div}(a(x,D\psi_2))|\right]dx \\
&\leqslant& c\frac{\alpha^{1-\xi}}{\xi-1}\left[ \int_{B_{R}(x_0)}|f|dx+D\Psi_2(B_{R}(x_0)) \right].
\end{eqnarray*}
Combining the last two estimates, the proof is complete.
\end{proof}
Analogous to the implications of Corollary \ref{dudw1c}, the subsequent Corollary is as follows.
\begin{corollary}\label{coro1}
Under the hypothesis of Corollary \ref{dudw1c}, then there exists $c=c(data)$ such that
\begin{equation*}
\fint_{B_{R}(x_0)}\frac{|V_{g}(Du)-V_{g}(Dw_0)|^{2}}{(\alpha+|u-w_0|)^{\xi}}dx\leqslant c\frac{\alpha^{1-\xi}}{\xi-1}\left[ \frac{|\mu|(\overline{B_{R}(x_0)})}{R^n}+\frac{D\Psi_2(B_{R}(x_0))}{R^n}\right]
\end{equation*}
for $\alpha>0$ and $\xi>1$.
\end{corollary}

\begin{lemma}\label{bjuw1}
Under the hypothesis of Lemma \ref{dudw1}, then there exists $c=c(data)$ such that
\begin{equation*}
\fint_{B_{R}(x_0)}\frac{|V_{g}(Dw_0)-V_{g}(Dw_1)|^{2}}{(\alpha+|w_0-w_1|)^{\xi}}dx\leqslant c\frac{\alpha^{1-\xi}}{\xi-1}\frac{D\Psi_2(B_{R}(x_0))}{R^n}
\end{equation*}
for $\alpha>0$ and $\xi>1$.
\end{lemma}

\begin{proof}
Let
$$\eta_{\pm}:=\frac{1}{\xi-1}\left[1-\left(1-\frac{(w_0-w_1)_{\pm}}{\alpha}\right)^{1-\xi}\right],$$
we test the inequality \eqref{bjw1-} and the equation \eqref{bjw1} with $v=w_0\pm\alpha \eta_{\mp}\geqslant \psi_1$ and $\varphi=w_1\pm\alpha\eta_{\pm}\geqslant \psi_1$ respectively, then
\begin{eqnarray*}
\int_{B_{R}(x_0)}\frac{|V_{g}(Dw_0)-V_{g}(Dw_1)|^{2}}{(\alpha+|w_0-w_1|)^{\xi}}dx
&\approx& \int_{B_{R}(x_0)} \frac{[a(x,Dw_0)-a(x,Dw_1)]\cdot(Dw_0-Dw_1)}{(\alpha+|w_0-w_1|)^{\xi}}dx \\
&\leq& c\int_{B_{R}(x_0)}\alpha^{1-\xi}(\eta_{+}+\eta_{-})|\operatorname{div}a(x,D\psi_2)|dx \\
&\leqslant& c\frac{\alpha^{1-\xi}}{\xi-1}D\Psi_2(B_{R}(x_0))
\end{eqnarray*}
and the proof is complete.
\end{proof}

Analogous to the demonstration of  Lemma \ref{bjuw1-} and Lemma \ref{bjuw1},  the subsequent lemma is presented.
\begin{lemma}\label{bjw1w2}
Under the hypothesis of Lemma \ref{dw1w2} and Lemma \ref{dw2w3}, then there exists $c=c(data)$ such that
\begin{equation*}
\fint_{B_{R}(x_0)}\frac{|V_{g}(Dw_1)-V_{g}(Dw_2)|^{2}}{(\alpha+|w_1-w_2|)^{\xi}}dx\leqslant c\frac{\alpha^{1-\xi}}{\xi-1}\frac{D\Psi_1(B_{R}(x_0))}{R^n}
\end{equation*}
\begin{equation*}
\fint_{B_{R}(x_0)}\frac{|V_{g}(Dw_2)-V_{g}(Dw_3)|^{2}}{(\alpha+|w_2-w_3|)^{\xi}}dx\leqslant c\frac{\alpha^{1-\xi}}{\xi-1}\frac{D\Psi_1(B_{R}(x_0))}{R^n}
\end{equation*}
for $\alpha>0$ and $\xi>1$.
\end{lemma}

Consulting Lemma 3.5, Lemma 3.6 and Lemma 3.7 in reference \cite{xiong3} leads us to the following lemma.
 \begin{lemma} \label{xsbj}
Suppose that the  assumptions of \eqref{a(x)1}, \eqref{a(x)3} and \eqref{dytj}  are satisfied, let  $B_{2R}(x_0)\subseteq \Omega$, $w_3\in W^{1,G}(B_{2R}(x_0))$ be the weak solution of \eqref{bjw3} and  $w_4\in W^{1,G}(B_{R}(x_0))$ be the weak solution of
 \begin{equation}\label{abar}
\left\{\begin{array}{r@{\ \ }c@{\ \ }ll}
\operatorname{div}\left(\overline{a}_{B_{R}(x_0)}(Dw_4)\right)&=&0  \ \ \ \ \ \mbox{in}\ \ B_{R}(x_0) \,, \\[0.05cm]
w_4&=&w_3  \ \ \ \mbox{on}\ \ \partial B_{R}(x_0) \,. \\[0.05cm]
\end{array} \right.
\end{equation}
\item[\rm (i)]Then there exists a constant $c=c(data)>0$ such that
 \begin{equation*}
 \fint_{B_{R}(x_0)}|Dw_3-Dw_4|dx\leqslant c\omega(R)^{\frac{1}{1+s_g}}\fint_{B_{2R}(x_0)}|Dw_3|dx.
 \end{equation*}
\item[\rm (ii)] Then there exist constants $\widehat{R}=\widehat{R}(data,\beta,\omega(\cdot))$ and $c=c(data,\beta)$ such that
 \begin{equation*}
 \parallel Dw_3\parallel_{L^{\infty}(B_{\frac{R}{2}}(x_0))}\leqslant c\fint_{B_{R}(x_0)}|Dw_3|dx
 \end{equation*}
 for every $ 0<R\leqslant \widehat{R}$ and $\beta$ is as in Lemma \ref{zcth}.

\item[\rm (iii)] For any $\sigma\in(0,1)$,  $0<R\leqslant \overline{R}=\overline{R}(data,\sigma,\omega(\cdot),c_0,\beta)$. If
 \begin{equation*}
 \sup_{B_{\frac{R}{2}}(x_0)}|Dw_3|\leqslant c_0 \lambda, \ \ \ \ where \ \ \ c_0\geqslant1, \ \lambda>0,
 \end{equation*}
then  there exists a constant $0<\overline{\delta}=\overline{\delta}(data,\sigma,c_0,\beta)<\frac{1}{300}$ such that
 \begin{equation*}
 osc_{B_{\overline{\delta}R}(x_0)}Dw_3\leqslant \sigma \lambda \ \ \ a.e.
 \end{equation*}
 where $\beta$ is as in Lemma \ref{zcth}.
 \end{lemma}

Next, assume that $x_0\in \Omega$ is the Lebesgue's point of $Du$, $B_{2R}(x_0)\subseteq \Omega$ and  we define
\begin{equation}\label{biii2}
B_R:=B_R(x_0), \ \ \ B_{i}:=B_{r_i}(x_0), \ \ \ \  r_i=\delta^{i}r, \ \ \ \ \
\end{equation}
\begin{equation*}
a_i:=|(Du)_{B_i}|=\vert\fint_{B_i}Du dx \vert, \ \ \ \ E_{i}:=E(Du,B_i)=\fint_{B_i}\vert Du-(Du)_{B_i} \vert dx,
\end{equation*}
where $\delta \in (0,\frac{1}{4})$, $0<r<\min\left\lbrace R, \overline{R}, \widehat{R}\right\rbrace$ will be determined later and $\overline{R}, \widehat{R}$ is as in Lemma \ref{xsbj}.
Moreover, assume that $u \in W^{1,1}(\Omega)$ with $\psi_2 \geqslant u\geqslant \psi_1$ a.e. is a limit of approximating solutions to $OP(\psi_1; \psi_2; \mu)$ with measure data $\mu \in \mathcal{M}_{b}(\Omega)$(in the sense of Definition \ref{opdy}), the sequence of functions  $w_0^i, w_1^i, w_2^i, w_3^i,w_4^i \in W^{1,G}(B_i)$ satisfy separately
\begin{equation*}
\left\{\begin{array}{r@{\ \ }c@{\ \ }ll}
&\int_{B_i}&  a(x, Dw_0^i)\cdot D(v-w_0^i)dx \geqslant \int_{B_i}  a(x, D\psi_2)\cdot D(v-w_0^i)dx  \,, \\[0.05cm]
& \mbox{for}&  \forall \ v \in w_0^i+W_{0}^{1,G}(B_i) \ \mbox{with} \ v\geqslant \psi_1 \  a.e. \ \mbox{in}\ \ B_i \,, \\[0.05cm]
&w_0^i&\geqslant \psi_1, \ \ \  \ \ \ \ \ \ \ \ a.e. \ \mbox{in}\ \ B_i\,, \\[0.05cm]
&w_0^i&=u \ \ \ \ \ \ \ \ \ \ \ \ \ \ \ \ \ \mbox{on}\ \ \partial B_i \,,
\end{array}\right.
\end{equation*}

\begin{equation*}
\left\{\begin{array}{r@{\ \ }c@{\ \ }ll}
&\int_{B_i}&  a(x, Dw_1^i)\cdot D(v-w_1^i)dx \geqslant0 \,, \\[0.05cm]
&\mbox{for} &\ \ \forall \ v \in w_1^i+W_{0}^{1,G}(B_i) \ \mbox{with} \ v\geqslant \psi_1 \  a.e. \ \mbox{in}\ \ B_i \,, \\[0.05cm]
&w_1^i&\geqslant \psi_1,  \ \ \  \ \ \ \ \ \ \ \ \ a.e. \ \mbox{in}\ \ B_i\,, \\[0.05cm]
&w_1^i&=w_0^i \ \ \ \ \ \ \ \ \ \ \ \ \ \ \ \ \ \mbox{on}\ \ \partial B_i \,,
\end{array}\right.
\end{equation*}

\begin{equation*}
\left\{\begin{array}{r@{\ \ }c@{\ \ }ll}
-\operatorname{div}\left( a(x,Dw_2^i)\right)&=&-\operatorname{div}\left( a(x,D\psi_1)\right) \ \ \ \ \mbox{in}\ \ B_i\,, \\[0.05cm]
w_2^i&=&w_1^i  \ \ \ \ \ \ \ \ \ \ \ \ \ \ \ \ \ \ \ \ \  \ \  \mbox{on}\ \ \partial B_i\,, \\[0.05cm]
\end{array}\right.
\end{equation*}

\begin{equation*}
\left\{\begin{array}{r@{\ \ }c@{\ \ }ll}
-\operatorname{div}\left( a(x,Dw_3^i)\right)&=&0 \ \ \ \ \mbox{in}\ \ B_i\,, \\[0.05cm]
w_3^i&=&w_2^i  \ \   \mbox{on}\ \ \partial B_i\,, \\[0.05cm]
\end{array}\right.
\end{equation*}

\begin{equation*}
\left\{\begin{array}{r@{\ \ }c@{\ \ }ll}
-\operatorname{div}\left(  \overline{a}_{B_i}(Dw_4^i)\right)&=&0  \ \ \ \  \ \mbox{in}\ \ \frac{1}{4}B_{i} \,, \\[0.05cm]
w_4^i&=&w_3^i  \ \ \ \mbox{on}\ \ \partial \frac{1}{4}B_{i} \,. \\[0.05cm]
\end{array}\right.
\end{equation*}

 Then we can obtain the following lemma.
\begin{lemma}\label{glam-0}
 Under the assumptions  \eqref{a(x)1} and \eqref{a(x)3}, suppose that for a certain index $i \in \mathbb{N}$ and for a number $\lambda>0$ there holds
 \begin{eqnarray}\label{tjgd}\nonumber
 &&A_{i-1}^{1}:=g^{-1}\left(\frac{|\mu|(\overline{B_{i-1})}}{r_{i-1}^{n-1}} \right)\leqslant \lambda, \ \ \  \ \ \
  A_{i-1}^{2}:=g^{-1}\left(\frac{D\Psi_1(B_{i-1})}{r_{i-1}^{n-1}} \right)\leqslant \lambda, \\
 && A_{i-1}^{3}:=g^{-1}\left(\frac{D\Psi_2(B_{i-1})}{r_{i-1}^{n-1}} \right)\leqslant \lambda, \ \ \ \ \  \frac{\lambda}{H}\leqslant |Dw_3^{i-1}|\leqslant H\lambda \ \ in \ B_i
 \end{eqnarray}
 for a constant $H\geqslant1$. Then there exists a constant $c=c(data,H,\delta)$ such that
\begin{equation*}
 \fint_{B_i}|Du-Dw_0^{i}|dx\leqslant c \frac{\delta^{-n}\lambda}{g(\lambda)}\left[ \frac{|\mu|(\overline{B_{i-1})}}{r_{i-1}^{n-1}}+ \frac{D\Psi_1(B_{i-1})}{r_{i-1}^{n-1}}+ \frac{D\Psi_2(B_{i-1})}{r_{i-1}^{n-1}}\right].
\end{equation*}
\end{lemma}
\begin{proof}
We start fixing the following quantities  $$2\chi =\frac{1}{2}\min \left\lbrace \frac{1}{s_g-1}, \frac{s_g}{(s_g-1)(n-1)},\frac{1}{n-1}\right\rbrace, \ \  \xi=1+2\chi $$ notice that $\xi <1^{*}=\frac{n}{n-1}$ and $\chi, \xi $ satisfy the conditions of Lemma \ref{dudw1-}.
From \eqref{tjgd},  it follows
\begin{eqnarray*}
&&\fint_{B_i}|Du-Dw_0^{i}|dx \\
&\leqslant& c \fint_{B_i}\frac{h_{\chi}(|Dw_3^{i-1}|)}{h_{\chi}(\lambda)}|Du-Dw_0^{i}|dx \\
&\leqslant& c \fint_{B_{i}} \frac{h_{\chi}(|Dw_0^{i}-Dw_3^{i-1}|)}{h_{\chi}(\lambda)}|Du-Dw_0^{i}|dx +c \fint_{B_{i}} \frac{h_{\chi}(|Dw_0^{i}|)}{h_{\chi}(\lambda)}|Du-Dw_0^{i}|dx \\
&:=&Q_1+Q_2.
\end{eqnarray*}
Our investigation begins by considering the estimation of  $Q_1$. We  utilize  \eqref{a(x)4}, Corollary \ref{dudw1c}, Lemma \ref{dudw1}, Lemma \ref{dw1w2}, Lemma \ref{dw2w3} and Young's inequality with conjugate functions $g_{\chi}$ and $g_{\chi}^{*}$ leading to
\begin{eqnarray*}
h_{\chi}(\lambda)Q_1 &\leqslant& c\fint_{B_i}g_{\chi}^{*}\left( \frac{g_{\chi}(|Dw_0^{i}-Dw_3^{i-1}|)}{|Dw_0^{i}-Dw_3^{i-1}|}\right)dx+c\fint_{B_i}g_{\chi}(|Du-Dw_0^{i}|)dx \\
&\leqslant& c\fint_{B_i}g_{\chi}(|Dw_0^{i}-Dw_3^{i-1}|)dx+c\fint_{B_i}g_{\chi}(|Du-Dw_0^{i}|)dx \\
&\leqslant& c\fint_{B_i}g_{\chi}(|Du-Dw_3^{i-1}|)dx+c\fint_{B_i}g_{\chi}(|Du-Dw_0^{i}|)dx \\
&\leqslant& c\fint_{B_i}g_{\chi}(|Du-Dw_0^{i}|)+g_{\chi}(|Du-Dw_0^{i-1}|)+g_{\chi}(|Dw_0^{i-1}-Dw_1^{i-1}|)\\
&+&g_{\chi}(|Dw_1^{i-1}-Dw_2^{i-1}|)
+g_{\chi}(|Dw_2^{i-1}-Dw_3^{i-1}|)dx \\
&\leqslant& c \delta^{-n}\left[ g_{\chi}(A_{i-1}^1)+g_{\chi}(A_{i-1}^2)+g_{\chi}(A_{i-1}^3)\right] .
\end{eqnarray*}
Then, by virtue of \eqref{tjgd} and the noted characteristic of $\frac{g(x)}{x}$ being a monotonically increasing function, we derive
\begin{eqnarray*}
Q_{1}&\leqslant& c\frac{\delta^{-n}\lambda}{g_{\chi}(\lambda)}\left[ g_{\chi}(A_{i-1}^1)+g_{\chi}(A_{i-1}^2)+g_{\chi}(A_{i-1}^3)\right]  \\
&=&c\frac{\delta^{-n}\lambda}{\left[ \frac{g(\lambda)}{\lambda}\right] ^{\chi}g(\lambda)}\left\lbrace \left[ \frac{g(A_{i-1}^1)}{A_{i-1}^1}\right] ^{\chi}g(A_{i-1}^1) +\left[ \frac{g(A_{i-1}^2)}{A_{i-1}^2}\right] ^{\chi}g(A_{i-1}^2) +\left[ \frac{g(A_{i-1}^3)}{A_{i-1}^3}\right] ^{\chi}g(A_{i-1}^3) \right\rbrace \\
&\leqslant& c\frac{\delta^{-n}\lambda}{g(\lambda)} \left[\frac{|\mu|(\overline{B_{i-1})}}{r_{i-1}^{n-1}}+\frac{D\Psi_1(B_{i-1})}{r_{i-1}^{n-1}} +\frac{D\Psi_2(B_{i-1})}{r_{i-1}^{n-1}} \right].
\end{eqnarray*}
 Next,  we proceed to evaluate $Q_2$, employing \eqref{vgz1} and  Corollary \ref{coro1} to obtain
\begin{eqnarray}\nonumber
h_{\chi}(\lambda)Q_2 &\leqslant& c \fint_{B_i}h_{\chi}(|Dw_0^{i}|)|Du-Dw_0^{i}|dx \\\nonumber
&\leqslant& c \fint_{B_i} \left[\frac{g(|Dw_0^{i}|)}{|Dw_0^{i}|} \right] ^{\frac{1+2\chi}{2}}|V_{g}(Du)-V_{g}(Dw_0^{i})|dx \\\nonumber
&\leqslant& c\fint_{B_i} \left[ \frac{|V_{g}(Du)-V_{g}(Dw_0^{i})|^{2}}{(\alpha+|u-w_0^{i}|)^{\xi}}\right] ^{\frac{1}{2}}[h_{2\chi}(|Dw_0^i|)(\alpha+|u-w_0^{i}|)^{\xi}]^{\frac{1}{2}}dx \\\nonumber
&\leqslant& c\left[\fint_{B_i}\frac{|V_{g}(Du)-V_{g}(Dw_0^{i})|^{2}}{(\alpha+|u-w_0^{i}|)^{\xi}} dx\right]^{\frac{1}{2}}\left[\fint_{B_i}h_{2\chi}(|Dw_0^{i})(\alpha+|u-w_0^{i}|)^{\xi}dx \right] ^{\frac{1}{2}} \\
&\leqslant& c\left[ \alpha^{1-\xi}\left( \frac{|\mu|(\overline{B_i})}{r_{i}^{n}}+\frac{D\Psi_2(B_i)}{r_{i}^{n}}\right) \right] ^{\frac{1}{2}}\left[\fint_{B_i}h_{2\chi}(|Dw_0^{i}|)(\alpha+|u-w_0^{i}|)^{\xi}dx \right] ^{\frac{1}{2}}, \label{hpL0}
\end{eqnarray}
where $\alpha>0$ to be determined.
By utilizing Corollary \ref{dudw1c}, Lemma \ref{dudw1}, Lemma \ref{dw1w2}, as well as  Lemma \ref{dw2w3} again, we  derive
\begin{eqnarray}\nonumber
\fint_{B_i}h_{2\chi}(|Dw_0^{i}|)dx&\leqslant & \fint_{B_i}h_{2\chi}(|Du-Dw_0^{i}|)+h_{2\chi}(|Du-Dw_0^{i-1}|)+h_{2\chi}(|Dw_0^{i-1}-Dw_1^{i-1}|)  \\\nonumber
&+&h_{2\chi}(|Dw_1^{i-1}-Dw_2^{i-1}|)+h_{2\chi}(|Dw_2^{i-1}-Dw_3^{i-1}|)+h_{2\chi}(|Dw_3^{i-1}|)dx  \\\nonumber
&\leqslant &ch_{2\chi}(\lambda)+c\delta^{-n}\left[ h_{2\chi}(A_{i-1}^1)+h_{2\chi}(A_{i-1}^2)+h_{2\chi}(A_{i-1}^3)\right]   \\
&\leqslant &c\delta^{-n}h_{2\chi}(\lambda).\label{h2p0}
\end{eqnarray}
Returning our attention to \eqref{hpL0}, we revisit
\begin{equation*}
Q_2\leqslant c\sqrt{\frac{\lambda}{g(\lambda)}}\left[\alpha^{1-\xi}\left( \frac{|\mu|(\overline{B_i})}{r_{i}^{n}}+\frac{D\Psi_2(B_i)}{r_{i}^{n}}\right) \right] ^{\frac{1}{2}}\left( \fint_{B_i}\frac{h_{2\chi}(|Dw_0^{i}|)}{h_{2\chi}(\lambda)}(\alpha+|u-w_0^{i}|)^{\xi}dx\right) ^{\frac{1}{2}}.
\end{equation*}
We choose
\begin{equation*}
\alpha=\left(\fint_{B_i}\frac{h_{2\chi}(|Dw_0^{i}|)}{h_{2\chi}(\lambda)}|u-w_0^{i}|^{\xi}dx \right) ^{\frac{1}{\xi}}+\sigma   \ \ \ \ \ \ \ \ for \ some \ \sigma>0.
\end{equation*}
Through the combination of \eqref{h2p0} with Young's inequality, it follows
\begin{eqnarray*}
Q_2&\leqslant& c\sqrt{\frac{\lambda}{g(\lambda)}}\left[\alpha^{1-\xi}\left( \frac{|\mu|(\overline{B_i})}{r_{i}^{n}}+\frac{D\Psi_2(B_i)}{r_{i}^{n}}\right) \right] ^{\frac{1}{2}}\left[\alpha^{\frac{\xi}{2}}\left( \fint_{B_i}\frac{h_{2\chi}(|Dw_0^{i}|)}{h_{2\chi}(\lambda)}dx\right) ^{\frac{1}{2}}+\alpha^{\frac{\xi}{2}} \right]  \\
&\leqslant& c\left[ \frac{\alpha}{r_i}\left( \frac{|\mu|(\overline{B_i})}{r_{i}^{n}}+\frac{D\Psi_2(B_i)}{r_{i}^{n}}\right)\frac{\lambda}{g(\lambda)}\right] ^{\frac{1}{2}}\left[\left( \fint_{B_i}\frac{h_{2\chi}(|Dw_0^{i}|)}{h_{2\chi}(\lambda)}dx\right) ^{\frac{1}{2}}+1 \right]  \\
&\leqslant& \varepsilon \frac{\alpha}{r_i}+c(\varepsilon)\left( \frac{|\mu|(\overline{B_i})}{r_{i}^{n}}+\frac{D\Psi_2(B_i)}{r_{i}^{n}}\right)\frac{\delta^{-n}\lambda}{g(\lambda)}.
\end{eqnarray*}
Ultimately, together with the estimation of
$Q_1$  to obtain
\begin{equation*}
\fint_{B_i}|Du-Dw_0^{i}|dx \leqslant \varepsilon \frac{\alpha}{r_i}+c\frac{\delta^{-n}\lambda}{g(\lambda)} \left[\frac{|\mu|(\overline{B_{i-1}})}{r_{i-1}^{n-1}}+\frac{D\Psi_1(B_{i-1})}{r_{i-1}^{n-1}} +\frac{D\Psi_2(B_{i-1})}{r_{i-1}^{n-1}} \right].
\end{equation*}
On the other hand, we estimate
\begin{eqnarray*}
\frac{g(\lambda)}{\lambda}\alpha &\leqslant & \left( \fint_{B_i}h_{2\chi}(|Dw_0^{i}|)|u-w_0^{i}|^{\xi}dx\right) ^{\frac{1}{\xi}}+\frac{g(\lambda)}{\lambda}\sigma \\
&\leqslant & \left( \fint_{B_i}h_{2\chi}(|Dw_3^{i-1}|)|u-w_0^{i}|^{\xi}dx\right) ^{\frac{1}{\xi}} \\
&+&\left( \fint_{B_i}h_{2\chi}(|Dw_0^{i}-Dw_3^{i-1}|)|u-w_0^{i}|^{\xi}dx\right) ^{\frac{1}{\xi}}+\frac{g(\lambda)}{\lambda}\sigma \\
&\leqslant & \mathbf{I}_{1}+\mathbf{I}_{2}+\frac{ g(\lambda)}{\lambda}\sigma.
\end{eqnarray*}
As for the estimate of $\mathbf{I}_{1}$, by \eqref{lg} to get $u-w_0^i \in W_0^{1,1}(B_i)$, then owing to the Sobolev's inequality, we have
\begin{equation*}
\frac{\mathbf{I}_{1}\lambda}{r_{i}g(\lambda)}\leqslant c\left(\fint_{B_i}\left|\frac{u-w_0^{i}}{r_i}\right|^{\xi}dx \right) ^{\frac{1}{\xi}}\leqslant c\fint_{B_i}|Du-Dw_0^{i}|dx.
\end{equation*}
For the estimate of $\mathbf{I}_{2}$, given the approximation $g(t)\approx f(t):=\int_{0}^{t}\frac{g(s)}{s}ds$ and the convexity of  $f(\cdot)$,  we assume the convexity of
$g(\cdot)$, establishing  $g(\cdot)$ is a Young function. Then Subsequently, we utilize Lemma \ref{gudu}, Corollary \ref{dudw1c}, Lemma \ref{dudw1}, Lemma \ref{dw1w2}, and Lemma \ref{dw2w3} for the estimation
\begin{eqnarray*}
\frac{\mathbf{I}_{2}}{r_i}&=&\left( \fint_{B_i}\left[ \frac{g(|Dw_0^{i}-Dw_3^{i-1}|)}{|Dw_0^{i}-Dw_3^{i-1}|}\frac{|u-w_0^{i}|}{r_i}\right] ^{\xi}dx\right) ^{\frac{1}{\xi}} \\
&\leqslant& c\left( \fint_{B_i}g^{*}\left(  \frac{g(|Dw_0^{i}-Dw_3^{i-1}|)}{|Dw_0^{i}-Dw_3^{i-1}|}\right)^{\xi} dx \right) ^{\frac{1}{\xi}}+c\left( \fint_{B_i}g\left(\frac{|u-w_0^{i}|}{r_i} \right)^{\xi}dx \right) ^{\frac{1}{\xi}} \\
&\leqslant& c\left( \fint_{B_i}g(|Dw_0^{i}-Dw_3^{i-1}|)^{\xi}dx\right) ^{\frac{1}{\xi}} +c\fint_{B_i}g(|Du-Dw_0^{i}|)dx \\
&\leqslant&c\fint_{B_i}g(|Du-Dw_0^{i}|)dx+c\left( \fint_{B_i}g(|Du-Dw_0^{i}|)^{\xi}dx\right) ^{\frac{1}{\xi}}\\
&+&c\left( \fint_{B_i}g(|Du-Dw_0^{i-1}|)^{\xi}dx\right) ^{\frac{1}{\xi}}+c\left( \fint_{B_i}g(|Dw_0^{i-1}-Dw_1^{i-1}|)^{\xi}dx\right) ^{\frac{1}{\xi}}\\
&+&c\left( \fint_{B_i}g(|Dw_1^{i-1}-Dw_2^{i-1}|)^{\xi}dx\right) ^{\frac{1}{\xi}} +c\left( \fint_{B_i}g(|Dw_2^{i-1}-Dw_3^{i-1}|)^{\xi}dx\right) ^{\frac{1}{\xi}}  \\
&\leqslant& c\delta^{-n} \left[ \frac{|\mu|(\overline{B_{i-1}})}{r_{i-1}^{n-1}}+\frac{D\Psi_1(B_{i-1})}{r_{i-1}^{n-1}}+\frac{D\Psi_2(B_{i-1})}{r_{i-1}^{n-1}}\right] .
\end{eqnarray*}
In conclusion, merging all estimates gives
\begin{eqnarray*}
&&\fint_{B_i}|Du-Dw_0^{i}|dx \\
&\leqslant& \varepsilon \fint_{B_i}|Du-Dw_0^{i}|dx+c\delta^{-n}\left[ \frac{|\mu|(\overline{B_{i-1}})}{r_{i-1}^{n-1}}+\frac{D\Psi_1(B_{i-1})}{r_{i-1}^{n-1}}+\frac{D\Psi_2(B_{i-1})}{r_{i-1}^{n-1}}\right] \frac{\lambda}{g(\lambda)}+\frac{\varepsilon\sigma}{r_i}.
\end{eqnarray*}
Now let $\sigma\rightarrow0$ and $\varepsilon=\frac{1}{2}$, we have
\begin{equation*}
\fint_{B_i}|Du-Dw_0^{i}|dx\leqslant c\frac{\delta^{-n}\lambda}{g(\lambda)}\left[ \frac{|\mu|(\overline{B_{i-1}})}{r_{i-1}^{n-1}}+\frac{D\Psi_1(B_{i-1})}{r_{i-1}^{n-1}}+\frac{D\Psi_2(B_{i-1})}{r_{i-1}^{n-1}}\right],
\end{equation*}
 which finishes our proof.
\end{proof}

\begin{lemma}\label{glam-}
  Under the same assumptions of Lemma \ref{glam-0}, then we have
\begin{equation*}
 \fint_{B_i}|Dw_0^{i}-Dw_1^{i}|dx\leqslant c \frac{\delta^{-n}\lambda}{g(\lambda)}\left[ \frac{|\mu|(\overline{B_{i-1})}}{r_{i-1}^{n-1}}+ \frac{D\Psi_1(B_{i-1})}{r_{i-1}^{n-1}}+ \frac{D\Psi_2(B_{i-1})}{r_{i-1}^{n-1}}\right],
\end{equation*}
where $c=c(data,H,\delta)$.
\end{lemma}
\begin{proof}
Since the proof is similar to that of Lemma \ref{glam-0}, we will only highlight the main points. Let $\chi, \xi $ are as in  Lemma \ref{glam-0}.
Then from \eqref{tjgd} we know
\begin{eqnarray*}
&&\fint_{B_i}|Dw_0^{i}-Dw_1^{i}|dx \\
&\leqslant& c \fint_{B_{i}} \frac{h_{\chi}(|Dw_1^{i}-Dw_3^{i-1}|)}{h_{\chi}(\lambda)}|Dw_0^{i}-Dw_1^{i}|dx +c \fint_{B_{i}} \frac{h_{\chi}(|Dw_1^{i}|)}{h_{\chi}(\lambda)}|Dw_0^{i}-Dw_1^{i}|dx \\
&:=&Q_1+Q_2.
\end{eqnarray*}
As for the estimate of $Q_1$. We apply \eqref{a(x)4}, Corollary \ref{dudw1c}, Lemma \ref{dudw1}, Lemma \ref{dw1w2}, Lemma \ref{dw2w3} and Young's inequality to get
\begin{eqnarray*}
h_{\chi}(\lambda)Q_1 &\leqslant& c\fint_{B_i}g_{\chi}^{*}\left( \frac{g_{\chi}(|Dw_1^{i}-Dw_3^{i-1}|)}{|Dw_1^{i}-Dw_3^{i-1}|}\right)dx+c\fint_{B_i}g_{\chi}(|Dw_0^{i}-Dw_1^{i}|)dx \\
&\leqslant& c\fint_{B_i}g_{\chi}(|Dw_0^{i}-Dw_1^{i}|)+g_{\chi}(|Du-Dw_0^{i-1}|)+g_{\chi}(|Dw_0^{i-1}-Dw_1^{i-1}|)\\
&+&g_{\chi}(|Dw_1^{i-1}-Dw_2^{i-1}|)+g_{\chi}(|Dw_2^{i-1}-Dw_3^{i-1}|)dx \\
&\leqslant& c \delta^{-n}\left[ g_{\chi}(A_{i-1}^1)+g_{\chi}(A_{i-1}^2)+g_{\chi}(A_{i-1}^3)\right] .
\end{eqnarray*}
Then by \eqref{tjgd}, we have
\begin{eqnarray*}
Q_{1} 
&\leqslant& c\frac{\delta^{-n}\lambda}{g(\lambda)} \left[\frac{|\mu|(\overline{B_{i-1})}}{r_{i-1}^{n-1}}+\frac{D\Psi_1(B_{i-1})}{r_{i-1}^{n-1}} +\frac{D\Psi_2(B_{i-1})}{r_{i-1}^{n-1}} \right].
\end{eqnarray*}
Next,  we estimate $Q_2$. Employing \eqref{vgz1} and Lemma \ref{bjuw1} to get
\begin{eqnarray}\nonumber
h_{\chi}(\lambda)Q_2 
&\leqslant& c \fint_{B_i} \left[\frac{g(|Dw_1^{i}|)}{|Dw_1^{i}|} \right] ^{\frac{1+2\chi}{2}}|V_{g}(Dw_0^{i})-V_{g}(Dw_1^{i})|dx \\\nonumber
&\leqslant& c\left[\fint_{B_i}\frac{|V_{g}(Dw_0^{i})-V_{g}(Dw_1^{i})|^{2}}{(\alpha+|w_0^{i}-w_1^{i}|)^{\xi}} dx\right]^{\frac{1}{2}}\left[\fint_{B_i}h_{2\chi}(|Dw_1^{i})(\alpha+|w_0^{i}-w_1^{i}|)^{\xi}dx \right] ^{\frac{1}{2}} \\
&\leqslant& c\left[ \alpha^{1-\xi}\frac{D\Psi_2(B_i)}{r_{i}^{n}}\right] ^{\frac{1}{2}}\left[\fint_{B_i}h_{2\chi}(|Dw_1^{i}|)(\alpha+|w_0^{i}-w_1^{i}|)^{\xi}dx \right] ^{\frac{1}{2}}, \label{hpL}
\end{eqnarray}
where $\alpha>0$ to be determined.
Using Corollary \ref{dudw1c}, Lemma \ref{dudw1}, Lemma \ref{dw1w2}, Lemma \ref{dw2w3} again, we have
\begin{eqnarray}\nonumber
&&\fint_{B_i}h_{2\chi}(|Dw_1^{i}|)dx \\\nonumber
&\leqslant & \fint_{B_i}h_{2\chi}(|Du-Dw_0^{i}|)+h_{2\chi}(|Dw_0^{i}-Dw_1^{i}|)+h_{2\chi}(|Du-Dw_0^{i-1}|)+h_{2\chi}(|Dw_0^{i-1}-Dw_1^{i-1}|) \\\nonumber
&+&h_{2\chi}(|Dw_1^{i-1}-Dw_2^{i-1}|)+h_{2\chi}(|Dw_2^{i-1}-Dw_3^{i-1}|)+h_{2\chi}(|Dw_3^{i-1}|)dx  \\\nonumber
&\leqslant &ch_{2\chi}(\lambda)+c\delta^{-n}\left[ h_{2\chi}(A_{i-1}^1)+h_{2\chi}(A_{i-1}^2)+h_{2\chi}(A_{i-1}^3)\right]   \\
&\leqslant &c\delta^{-n}h_{2\chi}(\lambda).\label{h2p}
\end{eqnarray}
Now we come back to \eqref{hpL}
\begin{equation*}
Q_2\leqslant c\sqrt{\frac{\lambda}{g(\lambda)}}\left[\alpha^{1-\xi}\frac{D\Psi_2(B_i)}{r_{i}^{n}} \right] ^{\frac{1}{2}}\left( \fint_{B_i}\frac{h_{2\chi}(|Dw_1^{i}|)}{h_{2\chi}(\lambda)}(\alpha+|w_0^i-w_1^{i}|)^{\xi}dx\right) ^{\frac{1}{2}}.
\end{equation*}
We take
\begin{equation*}
\alpha=\left(\fint_{B_i}\frac{h_{2\chi}(|Dw_1^{i}|)}{h_{2\chi}(\lambda)}|w_0^i-w_1^{i}|^{\xi}dx \right) ^{\frac{1}{\xi}}+\sigma   \ \ \ \ \ \ \ \ for \ some \ \sigma>0.
\end{equation*}
By combining \eqref{h2p} with Young's inequality gives
\begin{eqnarray*}
Q_2 
&\leqslant& \varepsilon \frac{\alpha}{r_i}+c(\varepsilon)\frac{D\Psi_2(B_i)}{r_{i}^{n-1}}\frac{\delta^{-n}\lambda}{g(\lambda)}.
\end{eqnarray*}
Finally, we combine with the estimate of $Q_1$ to get
\begin{equation*}
\fint_{B_i}|Dw_0^{i}-Dw_1^{i}|dx \leqslant \varepsilon \frac{\alpha}{r_i}+c\frac{\delta^{-n}\lambda}{g(\lambda)} \left[\frac{|\mu|(\overline{B_{i-1}})}{r_{i-1}^{n-1}}+\frac{D\Psi_1(B_{i-1})}{r_{i-1}^{n-1}} +\frac{D\Psi_2(B_{i-1})}{r_{i-1}^{n-1}} \right].
\end{equation*}
On the other hand, we estimate
\begin{eqnarray*}
\frac{g(\lambda)}{\lambda}\alpha 
&\leqslant & \left( \fint_{B_i}h_{2\chi}(|Dw_3^{i-1}|)|w_0^{i}-w_1^{i}|^{\xi}dx\right) ^{\frac{1}{\xi}} \\
&+&\left( \fint_{B_i}h_{2\chi}(|Dw_1^{i}-Dw_3^{i-1}|)|w_0^{i}-w_1^{i}|^{\xi}dx\right) ^{\frac{1}{\xi}}+\frac{g(\lambda)}{\lambda}\sigma \\
&\leqslant & \mathbf{I}_{1}+\mathbf{I}_{2}+\frac{ g(\lambda)}{\lambda}\sigma.
\end{eqnarray*}
For the estimate of $\mathbf{I}_{1}$, by   the Sobolev's inequality, we obtain
\begin{equation*}
\frac{\mathbf{I}_{1}\lambda}{r_{i}g(\lambda)}\leqslant c\left(\fint_{B_i}\left|\frac{w_0^{i}-w_1^{i}}{r_i}\right|^{\xi}dx \right) ^{\frac{1}{\xi}}\leqslant c\fint_{B_i}|Dw_0^{i}-Dw_1^{i}|dx.
\end{equation*}
As for the estimate of $\mathbf{I}_{2}$, we make use of Lemma \ref{gudu}, Corollary \ref{dudw1c}, Lemma \ref{dw1w2} and Lemma \ref{dw2w3} to estimate
\begin{eqnarray*}
\frac{\mathbf{I}_{2}}{r_i}
&\leqslant& c\left( \fint_{B_i}g^{*}\left(  \frac{g(|Dw_1^{i}-Dw_3^{i-1}|)}{|Dw_1^{i}-Dw_3^{i-1}|}\right)^{\xi} dx \right) ^{\frac{1}{\xi}}+c\left( \fint_{B_i}g\left(\frac{|w_0^{i}-w_1^{i}|}{r_i} \right)^{\xi}dx \right) ^{\frac{1}{\xi}} \\
&\leqslant&c\fint_{B_i}g(|Dw_0^{i}-Dw_1^{i}|)dx+c\left( \fint_{B_i}g(|Du-Dw_0^{i}|)^{\xi}dx\right) ^{\frac{1}{\xi}}+c\left( \fint_{B_i}g(|Dw_0^{i}-Dw_1^{i}|)^{\xi}dx\right) ^{\frac{1}{\xi}}\\
&+&c\left( \fint_{B_i}g(|Du-Dw_0^{i-1}|)^{\xi}dx\right) ^{\frac{1}{\xi}}+c\left( \fint_{B_i}g(|Dw_0^{i-1}-Dw_1^{i-1}|)^{\xi}dx\right) ^{\frac{1}{\xi}} \\
&+&c\left( \fint_{B_i}g(|Dw_1^{i-1}-Dw_2^{i-1}|)^{\xi}dx\right) ^{\frac{1}{\xi}}+c\left( \fint_{B_i}g(|Dw_2^{i-1}-Dw_3^{i-1}|)^{\xi}dx\right) ^{\frac{1}{\xi}}  \\
&\leqslant& c\delta^{-n} \left[ \frac{|\mu|(\overline{B_{i-1}})}{r_{i-1}^{n-1}}+\frac{D\Psi_1(B_{i-1})}{r_{i-1}^{n-1}}+\frac{D\Psi_2(B_{i-1})}{r_{i-1}^{n-1}}\right] .
\end{eqnarray*}
Finally, combining with all estimates to get
\begin{eqnarray*}
&&\fint_{B_i}|Dw_0^{i}-Dw_1^{i}|dx \\
&\leqslant &\varepsilon \fint_{B_i}|Dw_0^{i}-Dw_1^{i}|dx+c\delta^{-n}\left[ \frac{|\mu|(\overline{B_{i-1}})}{r_{i-1}^{n-1}}+\frac{D\Psi_1(B_{i-1})}{r_{i-1}^{n-1}}+\frac{D\Psi_2(B_{i-1})}{r_{i-1}^{n-1}}\right] \frac{\lambda}{g(\lambda)}+\frac{\varepsilon\sigma}{r_i}.
\end{eqnarray*}
Now let $\sigma\rightarrow0$ and $\varepsilon=\frac{1}{2}$, we obtain
\begin{equation*}
\fint_{B_i}|Dw_0^{i}-Dw_1^{i}|dx\leqslant c\frac{\delta^{-n}\lambda}{g(\lambda)}\left[ \frac{|\mu|(\overline{B_{i-1}})}{r_{i-1}^{n-1}}+\frac{D\Psi_1(B_{i-1})}{r_{i-1}^{n-1}}+\frac{D\Psi_2(B_{i-1})}{r_{i-1}^{n-1}}\right],
\end{equation*}
 which finishes our proof.
\end{proof}
The proof strategy of the following lemma is similar to Lemma \ref{glam-0} and Lemma \ref{glam-}, with the key distinction being the utilization of Lemma \ref{bjw1w2} in the proof process.
\begin{lemma}\label{glam-2}
 Under the same assumptions of Lemma \ref{glam-0}, then we  have
\begin{equation*}
 \fint_{B_i}|Dw_1^{i}-Dw_2^{i}|+|Dw_2^{i}-Dw_3^{i}|dx\leqslant c \frac{\delta^{-n}\lambda}{g(\lambda)}\left[ \frac{|\mu|(\overline{B_{i-1})}}{r_{i-1}^{n-1}}+ \frac{D\Psi_1(B_{i-1})}{r_{i-1}^{n-1}}+ \frac{D\Psi_2(B_{i-1})}{r_{i-1}^{n-1}}\right],
\end{equation*}
where $c=c(data,H,\delta)$.
\end{lemma}

\section{The proof of main theorem}\label{section4}
This section is dedicated to establishing the proofs of several main theorems.

\begin{proof}[Proof of Theorem \ref{th1}]
We define the quantity
\begin{equation*}
\lambda:=g^{-1}\left[ H_1 g\left( \fint_{B_{R}}|Du|dx\right)+H_{2}\mathbf{\RNum{1}}^{|\mu|}_{1}(x_0,2R)+H_{3}\mathbf{\RNum{1}}^{[\psi_1]}_{1}(x_0,2R) +H_{4}\mathbf{\RNum{1}}^{[\psi_2]}_{1}(x_0,2R) \right]
\end{equation*}
where the constants $H_1,H_2,H_3,H_4$ will be determined subsequently.  It's  our aim to establish that
\begin{equation}\label{dula}
|Du(x_0)|\leqslant \lambda.
\end{equation}
Without loss of generality we may assume  $\lambda>0$, otherwise \eqref{dula}  trivially follows from the monotonicity of the vector field.
We then  define
\begin{equation*}
C_{i}:=\sum_{j=i-2}^{i}\fint_{B_j}|Du|dx+\delta^{-n}E(Du,B_i), \ \ i\geqslant2, \ i\in \mathbb{N}.
\end{equation*}
Making use of  Lemma \ref{ag} to obtain
\begin{equation*}
C_2+C_3\leqslant 10\left(\frac{R}{r\delta^3} \right) ^{n}\delta^{-n}\fint_{B_R}|Du|dx \leqslant 10\delta^{-4n}H_{1}^{-\frac{1}{s_g}}\lambda\left(\frac{R}{r} \right) ^n.
\end{equation*}
We choose $H_{1}=H_{1}(data,\delta,r)$ large  enough to derive
\begin{equation*}
10\delta^{-4n}H_{1}^{-\frac{1}{s_g}}\left(\frac{R}{r} \right) ^n\leqslant\frac{1}{10},
\end{equation*}
then it follows
$$C_2+C_3\leqslant \frac{\lambda}{10}.$$
Without of generality, we can assume  there exists an exit time index $i_e\geqslant3$ such that
\begin{equation}\label{ciee2}
  C_{i_e}\leqslant \frac{\lambda}{10} \ \ \ \ but \ \ \ \ C_i>\frac{\lambda}{10}, \ \ for \ \ i>i_e.
\end{equation}
Otherwise, we would have $C_{i_j}\leqslant \frac{\lambda}{10} $ for an increasing subsequence $\left\lbrace i_j\right\rbrace $, we obtain
 $$|Du(x_0)|\leqslant \lim_{j\rightarrow \infty}\fint_{B_{i_j}}|Du|dx \leqslant \frac{\lambda}{10}.$$
Subsequently, our goal is to establish through induction that
\begin{equation}\label{aei2}
\fint_{B_i}|Du|dx \leqslant \lambda, \ \ \ \\ \ \forall i\geqslant i_e.
\end{equation}
Suppose that \eqref{aei2} is valid for $j=i_e,i_{e}+1,...,i$. Because of
 \begin{equation*}
C_{i_e}:=\sum_{j=i_{e}-2}^{i_e}\fint_{B_j}|Du|dx+\delta^{-n}E(Du,B_{i_{e}})\leqslant \frac{\lambda}{10},
\end{equation*}
we have
 \begin{equation*}
\fint_{B_j}|Du|dx \leqslant \lambda, \ \ \ \\ \ for \ j=i_e-2,...,i.
\end{equation*}
Thus, by utilizing Corollary \ref{dudw1c}, Lemma \ref{dudw1}, Lemma \ref{dw1w2}, Lemma \ref{dw2w3}, and Lemma \ref{xsbj}, we derive
\begin{eqnarray}\nonumber
\sup_{B_{j+1}}|Dw_3^j|&\leqslant& \sup_{\frac{1}{2}B_{j}}|Dw_3^j| \\\nonumber
&\leqslant& \fint_{B_j}|Du-Dw_0^j|+|Dw_0^j-Dw_1^j|+|Dw_1^j-Dw_2^j|+|Dw_2^j-Dw_3^j|+|Du|dx \\ \label{supdv2}
&\leqslant& c_{2}\left[ g^{-1}\left( \frac{|\mu|(\overline{B_j})}{r_{j}^{n-1}}\right)+g^{-1}\left( \frac{D\Psi_1(B_j)}{r_{j}^{n-1}}\right)+g^{-1}\left( \frac{D\Psi_2(B_j)}{r_{j}^{n-1}}\right)\right] +\lambda.
\end{eqnarray}
We calculate
 \begin{equation*}
 \sum_{i=0}^{+\infty}\frac{|\mu|(B_i)}{r_{i}^{n-1}}\leqslant \frac{2^{n-1}}{log2}\int_{r}^{2r}\frac{|\mu|(B_{\rho})}{\rho^{n-1}}\frac{d\rho}{\rho}+ \sum_{i=0}^{+\infty}\frac{1}{\delta^{n-1}log\frac{1}{\delta}}\int_{r_{i+1}}^{r_i}\frac{|\mu|(B_{\rho})}{\rho^{n-1}}\frac{d\rho}{\rho},
 \end{equation*}
it follows
 \begin{equation*}
\sum_{i=0}^{+\infty}\frac{|\mu|(B_i)}{r_{i}^{n-1}}\leqslant c_{1}\int_{0}^{2r}\frac{|\mu|(B_{\rho})}{\rho^{n-1}}\frac{d\rho}{\rho}\leqslant c_{1}\mathbf{\RNum{1}}^{|\mu|}_{1}(x_0,2R).
 \end{equation*}
Likewise,
 \begin{equation*}
\sum_{i=0}^{+\infty}\frac{D\Psi_1(B_i)}{r_{i}^{n-1}}\leqslant c_{1}\int_{0}^{2r}\frac{D\Psi_1(B_{\rho})}{\rho^{n-1}}\frac{d\rho}{\rho}\leqslant c_{1}\mathbf{\RNum{1}}^{[\psi_1]}_{1}(x_0,2R),
 \end{equation*}
  \begin{equation*}
\sum_{i=0}^{+\infty}\frac{D\Psi_2(B_i)}{r_{i}^{n-1}}\leqslant c_{1}\int_{0}^{2r}\frac{D\Psi_2(B_{\rho})}{\rho^{n-1}}\frac{d\rho}{\rho}\leqslant c_{1}\mathbf{\RNum{1}}^{[\psi_2]}_{1}(x_0,2R).
 \end{equation*}
Subsequent to Lemma \ref{ag} and with the definition of
$\lambda$  in mind, we deduce
\begin{equation}\label{g-1}
g^{-1}\left( \frac{|\mu|(\overline{B_j})}{r_{j}^{n-1}}\right)\leqslant g^{-1}\left( \sum_{i=0}^{+\infty}\frac{|\mu|(B_i)}{r_{i}^{n-1}}\right)\leqslant g^{-1}(c_{1}\mathbf{\RNum{1}}^{|\mu|}_{1}(x_0,2R)) \leqslant c_{1}^{\frac{1}{i_g}}H_{2}^{-\frac{1}{s_g}}\lambda,
\end{equation}
\begin{equation}\label{g-12}
g^{-1}\left( \frac{D\Psi_1(B_j)}{r_{j}^{n-1}}\right) \leqslant c_{1}^{\frac{1}{i_g}}H_{3}^{-\frac{1}{s_g}}\lambda,
\end{equation}
\begin{equation}\label{g-13}
g^{-1}\left( \frac{D\Psi_2(B_j)}{r_{j}^{n-1}}\right) \leqslant c_{1}^{\frac{1}{i_g}}H_{4}^{-\frac{1}{s_g}}\lambda.
\end{equation}
Consider  $H_{2}=H_{2}(data)$, $H_{3}=H_{3}(data)$ and $H_{4}=H_{4}(data)$  chosen sufficiently large so as to obtain
\begin{equation*}
c_{2}c_{1}^{\frac{1}{i_g}}\left( H_{2}^{-\frac{1}{s_g}}+H_{3}^{-\frac{1}{s_g}} +H_{4}^{-\frac{1}{s_g}} \right) \leqslant1.
\end{equation*}
Making use of  the last inequality together with \eqref{supdv2},  \eqref{g-1}, \eqref{g-12} and \eqref{g-13},  we derive
\begin{equation}\label{supdv}
\sup_{B_{j+1}}|Dw_3^j|\leqslant \sup_{\frac{1}{2}B_{j}}|Dw_3^j| \leqslant2\lambda.
\end{equation}
Subsequently,by using \eqref{supdv} and Lemma \ref{xsbj}  to get
\begin{eqnarray*}
\fint_{\frac{1}{4}B_j}|Dw_{4}^{j}|dx&\leqslant& \fint_{\frac{1}{4}B_j}|Dw_{3}^{j}|dx+\fint_{\frac{1}{4}B_j}|Dw_{4}^{j}-Dw_{3}^{j}|dx \\
&\leqslant& 2\lambda+c_{3}\omega(r_j)^{\frac{1}{1+s_g}}\fint_{\frac{1}{2}B_j}|Dw_{3}^{j}|dx \\
&\leqslant& c_{4}\lambda.
\end{eqnarray*}
For $m\geqslant3, \ m\in \mathbb{N}$  to be specified subsequently, we utilize   Lemma \ref{zcth} in combination with the last equation to obtain
\begin{equation*}
osc_{B_{j+m}}|Dw_{4}^{j}|\leqslant c_5\delta^{m\beta}\fint_{\frac{1}{4}B_j}|Dw_{4}^{j}|dx \leqslant \delta^{m\beta}c_{5}\lambda.
\end{equation*}
Assume $m=m(\delta,\beta,data)$ is taken sufficiently large to ensure
$$\delta^{m\beta}c_{5}\leqslant \frac{\delta^{n}}{200}.$$
Consequently, we obtain
\begin{equation}\label{oscb}
osc_{B_{j+m}}|Dw_4^j|\leqslant \frac{\delta^{n}}{200}\lambda.
\end{equation}
On the other hand, we employ Corollary \ref{dudw1c}, Lemma \ref{dudw1}, Lemma \ref{dw1w2}, and Lemma \ref{dw2w3} to obtain
\begin{eqnarray*}
&&\fint_{B_{j+m}}|Du-Dw_3^j|dx\\
 &\leqslant & \fint_{B_{j+m}}|Du-Dw_0^j|+|Dw_0^j-Dw_1^j|+|Dw_1^j-Dw_2^j|+|Dw_2^j-Dw_3^j|dx \\
 &\leqslant & c_{2}\delta^{-mn}\left[ g^{-1}\left(\frac{|\mu|(B_j)}{r_{j}^{n-1}} \right)+g^{-1}\left(\frac{D\Psi_1(B_j)}{r_{j}^{n-1}} \right) +g^{-1}\left(\frac{D\Psi_2(B_j)}{r_{j}^{n-1}} \right) \right]   \\
 &\leqslant& c_{2}\delta^{-mn}c_{1}^{\frac{1}{i_g}}\left( H_{2}^{-\frac{1}{s_g}}+H_{3}^{-\frac{1}{s_g}}+H_{4}^{-\frac{1}{s_g}}\right) \lambda.
\end{eqnarray*}
Subsequently, we choose $H_{2}=H_{2}(m,\delta,data)$,  $H_{3}=H_{3}(m,\delta,data)$ and $H_{4}=H_{4}(m,\delta,data)$ sufficiently large to obtain
$$c_{2}\delta^{-mn}c_{1}^{\frac{1}{i_g}}\left( H_{2}^{-\frac{1}{s_g}}+H_{3}^{-\frac{1}{s_g}}+H_{4}^{-\frac{1}{s_g}}\right)\leqslant \frac{\delta^n}{200}.$$
Therefore, we derive
\begin{equation}\label{bjm}
\fint_{B_{j+m}}|Du-Dw_3^j|dx\leqslant \frac{\delta^n}{200}\lambda.
\end{equation}
Next, making use of  the triangle inequality to get
\begin{eqnarray*}
&&\delta^{-n}\fint_{B_{j+m}}|Du-(Du)_{B_{j+m}}|dx \\
&\leqslant&2 \delta^{-n}\fint_{B_{j+m}}|Dw_4^j-(Dw_4^j)_{B_{j+m}}|+|Du-Dw_3^j|+|Dw_3^j-Dw_4^j|dx\\
&\leqslant&2\delta^{-n} osc_{B_{j+m}}|Dw_4^j|+2\delta^{-n}\fint_{B_{j+m}}|Du-Dw_3^j|dx\\
&+&\left( \frac{1}{4}\right) ^{n}2\delta^{-n-mn}c_{3}\omega(r_j)^{\frac{1}{1+s_g}}\fint_{\frac{1}{2}B_j}|Dw_3^j|dx .
\end{eqnarray*}
We reduce the value of $r$ -in a way depending on $m,\delta,data$- to gain
$$\left( \frac{1}{4}\right) ^{n}2\delta^{-n-mn}c_{3}\omega(r_j)^{\frac{1}{1+s_g}}\leqslant \frac{1}{200}.$$
Ultimately, invoking \eqref{supdv} in conjunction with \eqref{oscb} and \eqref{bjm} yields
\begin{equation*}
\delta^{-n}\fint_{B_{j+m}}|Du-(Du)_{B_{j+m}}|dx \leqslant\frac{\lambda}{20}.
\end{equation*}
Thanks to $m\geqslant 3$ and $j\geqslant i_{e}-2$, we get
\begin{equation*}
C_{j+m}=\sum_{k=j+m-2}^{j+m}\fint_{B_{k}}|Du|dx+\delta^{-n}E(Du,B_{j+m})> \frac{\lambda}{10}.
\end{equation*}
Therefore,
\begin{equation*}
\sum_{k=j+m-2}^{j+m}\fint_{B_{k}}|Du|dx> \frac{\lambda}{20}.
\end{equation*}
By employing this inequality together with \eqref{bjm}, we obtain
\begin{eqnarray*}
3 \sup_{B_{j+1}}|Dw_3^j|&\geqslant& \sum_{k=j+m-2}^{j+m}\fint_{B_k}|Dw_3^j|dx \\
&\geqslant& \sum_{k=j+m-2}^{j+m}\fint_{B_k}|Du|-|Du-Dw_3^j|dx \\
&\geqslant& \frac{\lambda}{20}-\frac{3\lambda}{200} \geqslant \frac{\lambda}{40}.
\end{eqnarray*}
Thus,  there exists a point $x_1\in B_{j+1}$ such that $Dw_{3}^j(x_1)>\frac{\lambda}{200}.$ Furthermore,  leveraging \eqref{supdv}, we can employ Lemma \ref{xsbj} with $\sigma=\frac{1}{1000}$. We select  $\delta>0$ sufficiently small so that $B_{j+1}\subseteq \overline{\delta}B_{j}$, where $\overline{\delta}=\overline{\delta}(data,\omega(\cdot),\beta)$ as defined in Lemma  \ref{xsbj}. In conclusion, we establish
$$osc_{B_{j+1}}|Dw_3^j|\leqslant \frac{\lambda}{1000}.$$
 Thus,  for any $x\in B_{j+1}$, we derive
\begin{equation*}
|Dw_3^j(x)|\geqslant |Dw_3^j(x_1)|-|Dw_3^j(x_1)-Dw_3^j(x)|\geqslant \frac{\lambda}{200}-\frac{\lambda}{1000}\geqslant \frac{\lambda}{1000}.
\end{equation*}
By \eqref{supdv}, we have
\begin{equation*}
\frac{\lambda}{1000}\leqslant |Dw_3^j|\leqslant 2\lambda \ \ \ \ in \ \ B_{j+1} \ \ \ \ for \ j=i_{e}-2,...,i.
\end{equation*}
Then using Lemma \ref{glam-0}, Lemma \ref{glam-} and  Lemma \ref{glam-2}, there exists $c_{6}=c_{6}(data)$ such that
\begin{eqnarray}\label{lamg} \nonumber
&&\fint_{B_{j+1}}|Du-Dw_0^{j+1}|+|Dw_0^{j+1}-Dw_1^{j+1}|+|Dw_1^{j+1}-Dw_2^{j+1}|+|Dw_2^{j+1}-Dw_3^{j+1}|dx \\
&\leqslant& c_{6}\frac{ \delta^{-n} \lambda}{g(\lambda)}\left[\frac{|\mu|(\overline{B_{j}})}{r_{j}^{n-1}} +\frac{D\Psi_1(B_{j})}{r_{j}^{n-1}} +\frac{D\Psi_2(B_{j})}{r_{j}^{n-1}} \right]
\end{eqnarray}
for $j=i_{e}-2,...,i$.
Next, we estimate
\begin{eqnarray}\label{zdgj-}\nonumber
&&E(Du, B_{j+1}) \\\nonumber
&\leqslant &2\fint_{B_{j+1}}|Dw_4^{j}-(Dw_4^{j})_{B_{j+1}}|dx+2\fint_{B_{j+1}}|Du-Dw_4^{j}|dx \\ \nonumber
&\leqslant &4^{\beta}2\delta^{\beta}\fint_{\frac{1}{4}B_{j}}|Dw_{4}^j-(Dw_4^j)_{\frac{1}{4}B_{j}}|dx\\ \nonumber
&+&2\fint_{B_{j+1}}|Du-Dw_0^{j}| +|Dw_0^{j}-Dw_1^{j}|+|Dw_1^{j}-Dw_2^{j}|+|Dw_2^{j}-Dw_3^{j}|+|Dw_3^{j}-Dw_4^{j}|dx \\   \nonumber
&\leqslant &4^{\beta+n+1}\delta^{\beta}\fint_{B_{j}}|Du-(Du)_{B_{j}}|dx+c_7 \delta^{-n}\fint_{\frac{1}{4}B_j}|Dw_3^j-Dw_4^j|dx \\
&+&c_7 \delta^{-n}\fint_{B_j}|Du-Dw_0^j|+|Dw_0^j-Dw_1^j|+|Dw_1^j-Dw_2^j|+|Dw_2^j-Dw_3^j|dx
\end{eqnarray}
for  $j=i_{e}-1,...,i+1$.
Now we proceed to  reduce the value of $\delta$ further in order to obtain
$$4^{\beta+n+1}\delta^{\beta}\leqslant \frac{1}{4}.$$
Therefore, thanks to \eqref{lamg} and Lemma \ref{xsbj}, we have
\begin{eqnarray*}
E(Du, B_{j+1}) &\leqslant& \frac{1}{4}E(Du, B_{j})\\
&+&c_{7}\delta^{-2n}\frac{\lambda}{g(\lambda)}\left[ \frac{|\mu|(\overline{B_{j-1}})}{r_{j-1}^{n-1}}
+\frac{D\Psi_1(B_{j-1})}{r_{j-1}^{n-1}}+\frac{D\Psi_2(B_{j-1})}{r_{j-1}^{n-1}}\right] +c_7\delta^{-n}\omega(r_j)^{\frac{1}{1+s_g}}\lambda
\end{eqnarray*}
for  $j=i_{e}-1,...,i+1$,
it follows
\begin{eqnarray*}
\sum_{j=i_e-1}^{i+2}E_j&\leqslant& E_{i_e-1}+\frac{1}{4}\sum_{j=i_e-1}^{i+1}E_j \\
&+&\frac{c_7}{\delta^{2n}}\frac{\lambda}{g(\lambda)}\sum_{j=0}^{+\infty}\left[ \frac{|\mu|(B_{j})}{r_{j}^{n-1}}+\frac{D\Psi_1(B_{j})}{r_{j}^{n-1}}+\frac{D\Psi_2(B_{j-1})}{r_{j-1}^{n-1}}\right]+c_7\delta^{-n}\lambda\sum_{j=0}^{+\infty}\omega(r_j)^{\frac{1}{1+s_g}} \\
&\leqslant&2 E_{i_e-1} +\frac{c_7}{\delta^{2n}}\frac{\lambda}{g(\lambda)}\sum_{j=0}^{+\infty}\left[ \frac{|\mu|(B_{j})}{r_{j}^{n-1}}+\frac{D\Psi_1(B_{j})}{r_{j}^{n-1}}+\frac{D\Psi_2(B_{j-1})}{r_{j-1}^{n-1}}\right] \\
&+&c_7\delta^{-n}\lambda\sum_{j=0}^{+\infty}\omega(r_j)^{\frac{1}{1+s_g}}.
\end{eqnarray*}
Subsequently, we proceed to estimate all the terms on the right-hand side of the inequality above.
\begin{equation*}
 \sum_{j=0}^{+\infty}\frac{|\mu|(B_j)}{r_{j}^{n-1}}\leqslant c_{1}\mathbf{\RNum{1}}^{|\mu|}_{1}(x_0,2R)) \leqslant \frac{c_1}{H_{2}}g(\lambda).
\end{equation*}
\begin{equation*}
 \sum_{j=0}^{+\infty}\frac{D\Psi_1(B_j)}{r_{j}^{n-1}}\leqslant c_{1}\mathbf{\RNum{1}}^{[\psi_1]}_{1}(x_0,2R)) \leqslant \frac{c_1}{H_{3}}g(\lambda).
\end{equation*}
\begin{equation*}
 \sum_{j=0}^{+\infty}\frac{D\Psi_2(B_j)}{r_{j}^{n-1}}\leqslant c_{1}\mathbf{\RNum{1}}^{[\psi_2]}_{1}(x_0,2R)) \leqslant \frac{c_1}{H_{4}}g(\lambda).
\end{equation*}
We further choose $H_2=H_2(n,\delta,i_g,s_g,l,L)$, $H_3=H_3(n,\delta,i_g,s_g,l,L)$ and $H_4=H_4(n,\delta,i_g,s_g,l,L)$ to be sufficiently large in order to obtain
$$\frac{c_7}{\delta^{2n}}\frac{c_1}{H_2}\leqslant\frac{\delta^n}{300}, \ \ \ \ \frac{c_7}{\delta^{2n}}\frac{c_1}{H_3}\leqslant\frac{\delta^n}{300}, \ \ \ \ \frac{c_7}{\delta^{2n}}\frac{c_1}{H_4}\leqslant\frac{\delta^n}{300}.$$
And by \eqref{dytj}, Wefurther proceed to  reduce the value of $r$-depending on $\delta,data$ such that
$$\sum_{j=0}^{+\infty}\omega(r_j)^{\frac{1}{1+s_g}}\leqslant c_{8}\int_{0}^{2r}\omega(\rho)^{\frac{1}{1+s_g}}\frac{d\rho}{\rho}\leqslant \frac{\delta^{2n}}{100c_7}.$$
Therefore, the inequalities stated above enable us to obtain
\begin{equation*}
\sum_{j=i_e-1}^{i+2}E_j\leqslant2 E_{i_e-1} +\frac{\lambda\delta^n}{50}\leqslant \frac{2}{5}\delta^{n}\lambda,
\end{equation*}
which implies
\begin{eqnarray*}
a_{i+1}&=&a_{i_e}+\sum_{j=i_e}^{i}(a_{j+1}-a_j) \\
&\leqslant& a_{i_e}+\sum_{j=i_e}^{i} \fint_{B_{j+1}}|Du-(Du)_{B_j}|dx \\
&\leqslant& \frac{\lambda}{10}+\frac{1}{\delta^n}\sum_{j=i_e}^{i}E_j \\
&\leqslant& \frac{2}{5}\lambda.
\end{eqnarray*}
Finally, we derive
\begin{eqnarray*}
\fint_{B_{i+1}}|Du|dx &\leqslant& \fint_{B_{i+1}}|Du-(Du)_{B_{i+1}}|+|(Du)_{B_{i+1}}|dx \\
&\leqslant& \frac{2}{5}\lambda+\frac{2}{5}\lambda \leqslant \frac{4}{5}\lambda.
\end{eqnarray*}
Therefore,  we obtain
$$|Du(x_0)|\leqslant \lim_{i\rightarrow \infty}\fint_{B_{i}}|Du|dx \leqslant \lambda.$$
Notably, the selection of parameters in the proof is feasible. Initially, we choose  $\delta$ to be sufficiently small, then we ensure that  $m=m(\delta)$  is sufficiently large, followed by selecting  $r$, which depends on both $m$ and $\delta$ to be suitably small. Finally, we set  $H_{1}=H_{1}(\delta,r)$, $H_{2}=H_{2}(m,\delta)$, $H_{3}=H_{3}(m,\delta)$ and $H_{4}=H_{4}(m,\delta)$ to be sufficiently large. With these choices, we conclude the proof of Theorem \ref{th1}.
\end{proof}
We now turn our attention to the demonstration of Theorem \ref{th2}. To be more specific, we will provide a brief outline of the proof of the subsequent Proposition \ref{prop1}, as with the potential estimate \eqref{gdex} in place, along with the Lemmas   proved in the previous  sections and the   Proposition \ref{prop1}, utilizing basic strategies extensively utilized in the preceding content,  this proof closely resembles the  Theorem 1.5 in \cite{km00}.
\begin{proposition}\label{prop1}
Suppose that the above assumptions of Theorem \ref{th1}   are satisfied, and moreover ,if
\begin{equation*}
\lim_{r\rightarrow 0}\frac{D\Psi_1(B_{r}(x))}{r^{n-1}}=\lim_{r\rightarrow 0}\frac{D\Psi_2(B_{r}(x))}{r^{n-1}}=\lim_{r\rightarrow 0}\frac{|\mu|(B_{r}(x))}{r^{n-1}}=0 \ \ \ \ \ \ \ \ \ locally \  uniformly \ in \ \Omega \ w.r.t. \ x,
\end{equation*}
then $Du$ is locally VMO-regular in $\Omega$. More precisely, for every $\varepsilon \in (0,1)$ and any open subsets $\Omega'\subset\subset \Omega'' \subset\subset \Omega$, there exists a radius $0<r_{\varepsilon}<dist( \Omega', \partial \Omega'')$, depending on $n,i_g,s_g,$ $v,L,M,$ $\mu(\cdot),$ $||Du||_{L^{\infty}(\Omega'')},$ $\omega(\cdot),$ $\varepsilon,\beta$ such that
\begin{equation}\label{brho}
\fint_{B_{\rho}(x_0)}|Du-(Du)_{B_{\rho}(x_0)}|dx \leqslant \varepsilon \lambda, \ \ \ \ \ \ \ \lambda:=||Du||_{L^{\infty}(\Omega'')}
\end{equation}
holds for $\rho \in (0,r_{\varepsilon})$ and $x_0\in \Omega'.$
\end{proposition}
\begin{proof}
For $x_0\in \Omega'$, we define
$$B_{i}:=B_{r_i}(x_0), \ \ \ \ \ \ r_{i}=\delta^{i}r, \ \ \ \ \ \ r\in (\delta R_{0}, R_{0}]$$
where $0<\delta<\frac{1}{2}, 0<R_{0}<dist(\Omega',\partial \Omega'')$ will be specified later.
We start by considering the definition of $\lambda$ and  the inclusion $B_{i}\subseteq \Omega''$, we have
\begin{equation*}
\fint_{B_i}|Du|dx\leqslant \lambda, \ \ \ \ \ \ for \ \forall \ i\in \mathbb{N}.
\end{equation*}
The aim is to establish that, for every $\varepsilon>0$,  it holds true that
\begin{equation}\label{elam}
E(Du,B_{i+2})\leqslant \varepsilon \lambda, \ \ \ \  i\in \mathbb{N}.
\end{equation}
Without of generality, we may assume that
\begin{equation*}
\fint_{B_{i+2}}|Du|dx\geqslant  \frac{\varepsilon \lambda}{2},
\end{equation*}
otherwise, \eqref{elam} is trival.

Next, let us select $R_{0}=R_{0}(data,\mu(\cdot),||Du||_{L^{\infty}(\Omega'')},\varepsilon,\delta,\omega(\cdot))$ to be sufficiently small  to obtain
\begin{eqnarray}\label{r0gj}\nonumber
&&\sup_{0<\rho<R_{0}}\sup_{x\in\Omega'}\left[ \frac{|\mu|(\overline{B_{\rho}}(x))}{\rho^{n-1}}+\frac{D\Psi_1(B_{\rho}(x))}{\rho^{n-1}} +\frac{D\Psi_2(B_{\rho}(x))}{\rho^{n-1}} \right] \leqslant g\left[ \frac{\varepsilon \lambda \delta^{2n}}{100c_1}\left( \frac{\delta^{n}}{10c_{4}}\right) ^{\frac{1}{i_g}}\right]  \\
&& \ \ \ \ \ \ \ \ \ \ \ \ \ \ \ \ \ \  \ \ \ \ \ \ \ \ \ \ \ \sup_{0<\rho<R_{0}} \omega(\rho)^{\frac{1}{1+s_g}}\leqslant \frac{\delta^{n}}{10c_{4}}.
\end{eqnarray}
Similar to the proof of Theorem \ref{th1}, we conclude that
\begin{eqnarray}\label{c3gj}\nonumber
\sup_{B_{i+1}}|Dw_3^{i}|&\leqslant& \sup_{\frac{1}{2}B_{i}}|Dw_3^{i}|\\ \nonumber
 &\leqslant& c_{1}\left[ g^{-1}\left( \frac{|\mu|(\overline{B_i})}{r_{i}^{n-1}}\right) +g^{-1}\left( \frac{D\Psi_1(B_i)}{r_{i}^{n-1}}\right)+g^{-1}\left( \frac{D\Psi_2(B_i)}{r_{i}^{n-1}}\right) \right]  +\lambda \\
 &\leqslant& c_{2}\lambda.
\end{eqnarray}
On the other hand, by using \eqref{r0gj}, Corollary \ref{dudw1c}, Lemma \ref{dudw1}, Lemma \ref{dw1w2} and Lemma \ref{dw2w3},  we have
\begin{eqnarray*}
\sup_{B_{i+1}}|Dw_3^{i}|&\geqslant& \fint_{B_{i+2}}|Dw_3^{i}|dx \\
&\geqslant& \fint_{B_{i+2}}|Du|dx-\fint_{B_{i+2}}|Du-Dw_1^{i}|+|Dw_1^{i}-Dw_2^{i}|+|Dw_2^{i}-Dw_3^{i}|dx \\
&\geqslant& \frac{\varepsilon \lambda}{2}-c_{1}\delta^{-2n}\left[ g^{-1}\left( \frac{|\mu|(\overline{B_i})}{r_{i}^{n-1}}\right)+\left( \frac{D\Psi_1(B_i)}{r_{i}^{n-1}}\right) +\left( \frac{D\Psi_2(B_i)}{r_{i}^{n-1}}\right)\right]   \\
&\geqslant& \frac{\varepsilon \lambda}{4}.
\end{eqnarray*}
So that there exists a point $x_{1}\in B_{i+1}$ such that
$$|Dw_3^{i}(x_1)|>\frac{\varepsilon \lambda}{4}.$$
Subsequently, we utilize Lemma \ref{xsbj} with $\sigma=\frac{\varepsilon}{100}$. We choose a sufficiently small  $\delta>0$  such that $B_{i+1}\subseteq \overline{\delta}B_{i}$, where $\overline{\delta}=\overline{\delta}(data,\omega(\cdot),\beta,\varepsilon)$as defined in Lemma  \ref{xsbj}. This yields,
$$osc_{B_{i+1}}|Dw_3^{i}|\leqslant \frac{\varepsilon \lambda}{100}.$$
Therefore,  for any $x\in B_{i+1}$, we have
\begin{equation*}
|Dw_3^{i}(x)|\geqslant |Dw_3^{i}(x_1)|-|Dw_3^{i}(x)-|Dw_3^{i}(x_1)|\geqslant \frac{\varepsilon \lambda}{8}.
\end{equation*}
The combination of \eqref{c3gj} with the preceding inequality yields
$$\frac{\varepsilon \lambda}{8}\leqslant |Dw_3^{i}|\leqslant c_{2}\lambda \ \ \ \ \ in \ \ B_{i+1}.$$
Thus,  all the assumptions of the Lemma \ref{glam-0}, Lemma \ref{glam-} and Lemma \ref{glam-2} are satisfied, then we  derive
\begin{eqnarray*}
&&\fint_{B_{i+1}}|Du-Dw_{0}^{i+1}|+|Dw_{0}^{i+1}-Dw_{1}^{i+1}|+|Dw_{1}^{i+1}-Dw_{2}^{i+1}|+|Dw_{2}^{i+1}-Dw_{3}^{i+1}|dx \\
&\leqslant &c_{3}\frac{\lambda}{g(\lambda)}\left[\frac{|\mu|(\overline{B_{i}})}{r_{i}^{n-1}}+\frac{D\Psi_1(B_{i})}{r_{i}^{n-1}}+\frac{D\Psi_2(B_{i})}{r_{i}^{n-1}} \right].
\end{eqnarray*}
Moreover, following the same procedure as in the calculation  of \eqref{zdgj-}, we have
\begin{eqnarray*}
E(Du,B_{i+2})&\leqslant& 4^{\beta+n+1}\delta^{\beta}E(Du,B_{i+1}) \\
&+&c_{4}\delta^{-n}\frac{\lambda}{g(\lambda)}\left[\frac{|\mu|(\overline{B_{i}})}{r_{i}^{n-1}}+\frac{D\Psi_1(B_{i})}{r_{i}^{n-1}} +\frac{D\Psi_2(B_{i})}{r_{i}^{n-1}}  \right]+c_{4}\delta^{-n}\omega(r_{i+1})^{\frac{1}{1+s_g}}\lambda.
\end{eqnarray*}
By choosing $\delta=\delta(data,\beta,\varepsilon)$ sufficiently small, we ensure that $$4^{\beta+n+1}\delta^{\beta}\leqslant \frac{\varepsilon}{4}.$$
Furthermore, by\eqref{r0gj}, we derive
\begin{equation*}
E(Du,B_{i+2})\leqslant \frac{\varepsilon}{4}E(Du,B_{i+1})+\frac{\varepsilon\lambda}{5}.
\end{equation*}
Consequently,  we derive \eqref{elam} by induction.
Finally, we choose $r_{\varepsilon}=\delta^{3}R_0$,  ensuring that for any $0<\rho<\delta^{3}R_0$, there exists an integer $m\geqslant3$ such that
$\delta^{m+1}R_0<\rho<\delta^{m}R_0$, which means that
$\rho=\delta^{m}r$ for some $r\in(\eta R_0,R_0]$ and \eqref{brho} follows from \eqref{elam}.Then we derive Propsition \ref{prop1}.
\end{proof}

\section*{Acknowledgments}The authors are supported by 
the Fundamental Research Funds for the Central Universities 
  (Grant No.  2682024CX028) and the  National Natural Science Foundation of China (Grant No.~12071229 and 12101452).

\end{document}